\documentclass[11pt]{amsart}
\usepackage[dvipsnames]{xcolor}
\usepackage[colorlinks,
	linkcolor=Red,
	anchorcolor=blue,
	citecolor=ForestGreen
]
{hyperref}

\usepackage[dvipsnames]{xcolor}
\usepackage[top=3cm, bottom=3cm, left=2.8cm, right=2.8cm]{geometry}
\usepackage{lmodern}  

\usepackage[T1]{fontenc}
\usepackage{textcomp}
\usepackage{amsmath, amssymb,mathrsfs}
\usepackage{mathtools}

\usepackage{tikz}
\usepackage{orcidlink}

\makeatletter
\newcommand*{\rom}[1]{\expandafter\@slowromancap\romannumeral #1@}
\makeatother

\renewcommand{\d}{\,\mathrm{d}}

\newcommand\N{\ensuremath{\mathbb{N}}}
\newcommand\R{\ensuremath{\mathbb{R}}}
\newcommand\Z{\ensuremath{\mathbb{Z}}}

\newcommand\C{\ensuremath{\mathbb{C}}}

\usepackage{bbm}
\newcommand{\un}{\ensuremath{\mathbbm{1}}}

\newtheorem{theorem}{Theorem}[section]
\newtheorem{proposition}[theorem]{Proposition}
\newtheorem{corollary}[theorem]{Corollary}
\newtheorem{lemma}[theorem]{Lemma}
\theoremstyle{definition}
\newtheorem{remark}[theorem]{Remark}

\newtheorem*{claim}{Claim}
\newtheorem{definition}[theorem]{Definition}
\newtheorem{taggedtheoremx}{Assumption}
\newenvironment{assumption}[1]
 {\renewcommand\thetaggedtheoremx{#1}\taggedtheoremx}
 {\endtaggedtheoremx}

\DeclarePairedDelimiterX\ipd[2]{\langle}{\rangle}{#1\delimsize , #2}

\numberwithin{equation}{section}

\mathtoolsset{showonlyrefs}

\colorlet{review1}{black}
\colorlet{review2}{black}
\colorlet{myself}{black}

\begin{document}
\title[Quantitative propagation of smallness and control]{Quantitative 2D propagation of smallness \\and control for 1D heat equations \\with power growth potentials}

\author[Y.~Wang]{Yunlei Wang\,\orcidlink{0000-0001-6711-2611}}\address{Yunlei Wang, Institut de Mathématiques de Bordeaux, UMR 5251, Université de Bordeaux, CNRS, Bordeaux INP, F-33400 Talence, France}
\email{yunlei.wang@math.u-bordeaux.fr}


\begin{abstract}
We study the relation between propagation of smallness in the plane and control for heat equations. The former has been proved by Zhu \cite{zhu2023remarks} who showed how the value of solutions in some small set propagates to a larger domain. By reviewing his proof, we establish a quantitative version with the explicit dependence of parameters. Using this explicit version, we establish new exact null-controllability results of 1D heat equations with any nonnegative power growth potentials $V\in L^\infty_{\mathrm{loc}}(\R)$. As a key ingredient, new spectral inequalities are established. The control set $\Omega$ that we consider satisfy
\begin{equation}
    \left|\Omega\cap [x-L\langle x\rangle ^{-s},x+L\langle x\rangle ^{-s}]\right|\ge \gamma^{\langle x\rangle^\tau}2L\langle x\rangle^{-s}
\end{equation}
for some $\gamma\in(0,1)$, $L>0$, $\tau,s\ge 0$, and $\langle x\rangle:=(1+|x|^2)^{1 /2} $. In particular, the null-controllability result for the case of thick sets that allow the decay of the density (\textit{i.e.}, $s=0$ and $\tau\ge 0$) is included. These extend the results in \cite{zhu2023spectral} from $\Omega$ being the union of equidistributive open sets to thick sets in the 1-dimensional case, and in \cite{su2023quantitative} from bounded potentials to certain unbounded ones.
\end{abstract}

\maketitle
\vspace*{-0.2\baselineskip}

{
\tableofcontents
}

\section{Introduction}
In this article, we study the relation between the propagation of smallness in the plane \cite[Theorem~1.1]{zhu2023remarks} and control for 1-dimensional heat equations with power growth potentials. The former helps us to understand the propagation of smallness of the finite sum of eigenfunctions of the Schrödinger operator $H=-\partial_x^2+V(x)$ on $\R$, that is, the spectral inequalities associated with $H$. Based on these new spectral inequalities, we give new control results for heat equations with power growth potentials. 

\subsection{Main results}
Consider the 1D heat equation,
\begin{equation}
	\begin{cases}
	\partial_t u-\partial_x^2 u +V(x) u=h(t,x)\un_{\Omega},\quad x\in \R,\,t>0,\\
	u|_{t=0}=u_0 \in L^2(\R),\label{heat}
	\end{cases}
\end{equation}
where potential $V$ is a real-valued nonnegative funtion, $h(t,x)\in L^2\left( (0,T)\times \R \right) $, and $\Omega\subset \R$ is a given measurable set. Equation~\eqref{heat} is said to be \textit{exactly null-controllable} (or simply \textit{null-controllable}) from the set $\Omega$ in time $T>0$ if, for any intial datum $u_0 \in L^2(\R)$, there exists $h(t,x)\in L^2\left( (0,T)\times \Omega \right) $ such that the mild solution of~\eqref{heat} satisfies $u(T)=0$.

The potential $V$ that we are concerned with in~\eqref{heat} satisfies the following assumption: 
\begin{assumption}{A1}\label{assump1}
	$V(x)\in L^\infty_{\mathrm{loc}}(\R)$ is a nonnegative real-valued potential and there exist constants $c_1> 0$, $c_2> 0$, $c_3>0$ and  $\beta_2\ge \beta_1> 0$ such that
	\[
	c_1(|x|-c_3)_+^{\beta_1}\le V(x)\le c_2\langle x\rangle ^{\beta_2}, \quad \forall x\in\R,
	\] 
	where $(a)_{+}:=\max\left\{0,a\right\}$ and  $\langle x\rangle :=(1+|x|^2)^{\frac{1}{2}}$ denotes the Japanese bracket.
\end{assumption}

\begin{remark}
    To see how general this assumption is, we consider the potential
    \begin{equation}
        V(x)=|x|^{\beta_2}(\sin(x^2)+1)+|x|^{\beta_1}.\label{exa}
    \end{equation}
It satisfies Assumption~\ref{assump1} and oscillates between $|x|^{\beta_1}$ and $(2+o(x))^{\beta_2}$ for $|x|\ge 1$. Besides, the oscillation frequency for $V(x)$ and $V'(x)$ tends to infinity when $|x|\to \infty$.
\end{remark}

We give some notations and definitions to describe the control sets $\Omega$ we are concerned with. Given any measurable set $\Omega\subset \R$, we denote by $|\Omega|$ its measure.

\begin{definition}\label{def-by-prof} 
Let $\rho(x):\R\to\R_+$ be a non-increasing function and $\tau\ge 0$ be a nonnegative real number. We call a set $\Omega\subset \R$ \textit{thick of type $(\rho,\tau)$}, if there exist positive constants $\gamma\in(0,1)$ and $L>0$, such that
\begin{equation}\label{eqn-by-prof}
    |\Omega\cap I_{L\rho(x)}(x)|\ge \gamma^{\langle x\rangle^\tau} |I_{L\rho(x)}(x)|
\end{equation}
for all $x\in\R$ where we denote $I_{r}(x):=[x-r,x+r]$. If needed, we will say that $\Omega$ is \textit{$(\gamma,L)$-thick of type $(\rho,\tau)$} if we want to highlight the dependence on $\gamma$ and $L$.
\end{definition}

The first type of functions $\rho$ we are concerned with is 
\begin{equation}\label{rho}
    \rho_s(x):=\frac{1}{\langle x\rangle^s}
\end{equation}
for some $s\ge 0$.

     For the above definition and $\rho_s$, there are two things which we need to mention here: 
     \begin{enumerate}
         \item If we fix $s=0$, the $(L,\gamma)$-thick set $\Omega$ of type $(\rho_0, \tau)$ is reduced to such a condition: there exist new constants $L>0$ and $\gamma>0$ and the same $\tau$ such that
         \begin{equation}
             |\Omega\cap [x,x+L]|\ge \gamma^{\tau} L,\quad \forall x\in\R.
         \end{equation}
         If we fix $\tau=0$ further, the $(L,\gamma)$-thick set $\Omega$ of type $(\rho_0,0)$ is just the usual definition of a \textit{thick set}, that is, there exist new constants $L>0$ and $\gamma>0$ such that
\begin{equation}
    |\Omega\cap [x,x+L]|\ge \gamma L,\quad \forall x\in\R.
\end{equation}
We also call the set $\Omega$ a \textit{$(L,\gamma)$-thick set} when we want to explicitly show the parameters. The definition of thick sets arose from studies of the uncertainty principle and the name was introduced in \cite{kovrijkine2001some}. Before \cite{kovrijkine2001some}, some very similar concepts like, \textit{e.g.}, \textit{relative dense sets} in \cite{kacnel1973equivalent}, were proposed.
         \item When $\rho=\rho_s$, we can relax the above definition by only assuming that $\Omega$ satisfies~\eqref{eqn-by-prof} for $|x|$ large enough. Indeed, let $A>0$ and $\Omega\subset \R$, if $\Omega$ satisfies ~\eqref{eqn-by-prof} with $\rho=\rho_s$ only for $|x|\ge A$, then we may choose a new $L$ large enough and a new $\gamma$ small enough such that $\Omega$ is still thick of type $(\rho_s,\tau)$ defined by Definition~\ref{def-by-prof}. This scaling approach is essentially the same as the proof of Lemma~\ref{equiv-lma} in Subsection~\ref{reduction-thick}.
     \end{enumerate}
     
Now we can state our first null-controllability result:
\begin{theorem}\label{main}
	Let $V$ be a potential satisfying Assumption~\ref{assump1} with $\beta_2\ge \beta_1>0$. Let $\Omega\subset \R$ be thick of type $(\rho_s,\tau)$ with $\rho_s$ defined in~\eqref{rho}, $s\ge \beta_2 /2$, $0\le \tau <\beta_1 /4$. Then Equation~\eqref{heat} is exactly null-controllable from $\Omega$ at any time $T>0$. 
\end{theorem}
In the critical case $s=\beta_2/2$, we obtain a better result than the general one given by Theorem~\ref{main}, in a sense that we replace $\rho_s$ defined in~\eqref{rho} by a function that decays slightly slower. 
We state our second null-controllability result: 
\begin{theorem}\label{main-2}
Let $V$ be a potential satisfying Assumption~\ref{assump1} with  $\beta_2\ge \beta_1>2$. Then there exists $R_0=R_0(V)>0$ depending only on $V$ such that if $0<R<R_0$ and $\Omega$ is a $(L,\gamma)$-thick set of type $(\rho,\tau)$ whith $\rho$ defined by
\begin{equation}\label{new-rho}
    \rho(x):=\min\left\lbrace \left(R\log\log\langle x\rangle-R^{-1}\right)\frac{1}{\langle x\rangle^{\frac{\beta_2}{2}}},1\right\rbrace
\end{equation} 
for $0\le \tau <(\beta_1-2)/4$, Equation~\eqref{heat} is exactly null-controllable from $\Omega$ at any time $T>0$.
\end{theorem}

On the other hand, we consider the case $s=0$ in~\eqref{eqn-by-prof}, that is, the density condition
\begin{equation}\label{eqn-by-Ming}
    |\Omega\cap [Ln,L(n+1)]|\ge \gamma^{\langle n\rangle^\tau}L
\end{equation}
for some $L>0$, $\gamma\in [0,1)$ and all $n\in\Z$.
\begin{definition}\label{def-thick-with-decaying-density}
Let $\Omega$ be a measurable set in $\R$. We call it the \textit{thick set with decaying density} if it satisfies~\eqref{eqn-by-Ming}. In particular, when $\tau=0$, the set $\Omega$ is \textit{$(L,\gamma)$-thick}. 
\end{definition}
For this kind of sets, we consider~\eqref{heat} with a potential $V$ that satisfies a bit more restrictive assumption:

\begin{assumption}{A2}\label{assump2}
    $V\in L^{\infty}_{\mathrm{loc}}(\R)$ is a nonnegative real-valued potential satisfing Assumption~\ref{assump1}. We say that $V$ satisfies Assumption~\ref{assump2} if it can be written as $V=V_1+V_2$ and
    \begin{equation}\label{regularity-condition}
        |V_1(x)|+|DV_1(x)|+|V_2(x)|^{\frac{4}{3}}\le c_4 \langle x\rangle^{\beta_2}
    \end{equation}
    for some positive constant $c_4$.
\end{assumption}

Under this assumption, we have the following result:
\begin{theorem}\label{main-3}
Let $V$ be a potential satisfying Assumption~\ref{assump2} with $2\beta_1>\beta_2\ge \beta_1>0$. Let $\Omega$ be a thick set with decaying density which satisfies~\eqref{eqn-by-Ming} for some $L>0$, $\gamma\in [0,1)$ and $0\le \tau<(2\beta_1-\beta_2)/4$. Then Equation~\eqref{heat} is exactly null-controllable from $\Omega$ at any time $T>0$.
\end{theorem}

The main novelty of these results is the absence of a regularity assumption on the control sets, with only a wild regularity assumption on power growth potentials.  While to our knowledge, previous results for rough control sets required real analiticity on power growth potentials. It is worth to mention that for heat equations with bounded potentials, the null-controllability from thick sets was proved by Su, Sun, and Yuan recently in \cite[Corollary~3.11]{su2023quantitative}. The results here generalize their result to unbounded potentials and allows the decaying density of the control sets. 

Let $\{e^{-tH}\}_{t\ge 0}$ be the semigroup generated by $-H=\partial_x^2-V(x)$.
By the Hilbert Uniqueness Method (see \cite[Theorem~2.44]{coron2007control} or \cite{lions1988controlabilite,tucsnak2009observation}), the null-controllability of~\eqref{heat} from the set $\Omega$ in time $T>0$ is equivalent to the inequality
\begin{equation}
	\|{\color{review1}e^{-HT}}u_0\|^2_{L^2(\R)}\le C(T,V,\Omega) \int_0^{T}\|e^{-Ht}u_0\|^2_{L^2(\Omega)}\d t, \quad \forall u_0 \in L^2(\R),\label{obs} 
\end{equation}
where $C(T,V,\Omega)$ is a constant which depends only on $T,V$ and $\Omega$. Then by the famous Lebeau-Robbiano method introduced in \cite{lebeau1995controle} (see also \cite{tenenbaum2011null,beauchard2018null,nakic2020sharp,gallaun2020sufficient}), the proof of~\eqref{obs}, and therefore the proofs of Theorem~\ref{main}, Theorem~\ref{main-2} and Theorem~\ref{main-3}, can be reduced to the proofs of new spectral inequalities, which we introduce now. 

\subsection{Spectral inequalities}\label{subsection-spectral-inequality}
\textcolor{review1}{Given any $d$-dimensional Euclidean space $\R^d$,} a \textit{spectral inequality} for a nonnegative selfadjoint operator $P$ in $L^2(\R^d)$ takes the form
\begin{equation}\label{general-spectral-inequality}
    \|\phi\|_{L^2(\R^d)}\le Ce^{C\lambda^{\zeta}}\|\phi\|_{L^2(\Omega)}
\end{equation}
for all $\lambda>0$ and all $\phi\in\mathcal{E}_\lambda(P)$,  where $\Omega$ is a measurable subset of $\R^d$, $\mathcal{E}_\lambda(P)=\un_{(-\infty,\lambda]}(A)$ are the spectral subspaces for the operator $P$ in $L^2(\R^d)$ associated with the interval $(-\infty,\lambda]$, and $C,\zeta$ are positive constants.

In our cases, we consider the 1D Schrödinger operator
\begin{equation}
	Hf(x):=H_{V}f(x):=-\partial_x^2f(x)+V(x)f(x),\quad \forall f \in D(H) 
\end{equation}
where $D(H)$ denotes the domain of the operator $H$:
\begin{equation}
	D(H)=\left\{f\in L^2(\R): \partial_xf \in L^2(\R) \text{ and }Vf\in L^2(\R)\right\}. 
\end{equation}

To make a comparison later, we consider it under two circumstances, first under Assumption~\ref{assump1}, and the second under the assumption that $V$ is bounded and non-negative. In both cases, the space of Schwartz functions, denoted by $\mathcal{S}(\R)$, is contained in $D(H)$. 

For $V$ satisfying Assumption~\ref{assump1}, the potential $V$ satisfies
\begin{equation}
	\lim_{|x| \to \infty} V(x)=\infty.
\end{equation}
This implies that the inverse operator $H^{-1}$ is compact in $L^2(\R)$ and therefore the spectrum of $H$ is discrete and unbounded. Precisely speaking, there exists a sequence of real numbers $\left\{\lambda_k\right\}_{k\in \N}$ with $0<\lambda_0\le \lambda_1\le\cdots $ and $\lambda_k\to \infty$, and an orthonormal basis $\left\{\phi_k\right\} _{k\in \N}$ of $L^2(\R)$, such that 
\begin{equation}
	H\phi_k=\lambda_k^2 \phi_k,\quad \forall k \in \N.
\end{equation}
The above identities also mean that, for each $k\in\N$, $\phi_k$ is the eigenfunction of the operator $H$ with eigenvalue $\lambda_k^2$.
Hence for any $\lambda>0$, the spectral subspace $\mathcal{E}_\lambda(\sqrt{H})$ can be recognized as the space of functions spanned by eigenfunctions of the operator $H$ with eigenvalues no larger than $\lambda^2$, that is,
\begin{equation}
	\mathcal{E}_\lambda(\sqrt{H})=\mathrm{span}\left\{\phi_k: k \text{ such that }\lambda_k\le \lambda\right\}. 
\end{equation}

For $V$ bounded and non-negative, things are slightly different. The operator $H$ on $L^2(\R)$ with domain $H^2(\R)$ is self-adjoint with non-negative continuous spectrum. Based on the spectral theorem (see for instance~\cite[Section~2.5]{davies1995spectral}), there exists a spectral measure $\d m_\lambda$ of the operator $\sqrt{H}$ such that
\begin{equation}
    f=\int_0^\infty \d m_\lambda f,\quad \forall f\in L^2(\R).
\end{equation}
Moreover, we have
\begin{equation}
    F(\sqrt{H}) =\int_0^\infty {\color{review1}F}(\lambda)\d m_\lambda, \quad \forall F\in L^\infty(\R),
\end{equation} 
which satisfies
\begin{equation}
    ({\color{review1}G}(\sqrt{H})f,{\color{review1}H}(\sqrt{H})f)=\int_0^\infty G(\lambda)\overline{H(\lambda)}\left(\d m_\lambda f,f\right)_{L^2(\R)}, \quad \forall f\in L^2(\R)
\end{equation}
for any $G,H\in L^\infty(\R)$.
Then the spectral projector $\Pi_\mu$ associated with the function $F(\lambda)=\un_{\lambda\le\mu}$, is defined by
\begin{equation}
    {\color{review1}\Pi_\mu f:=\int_0^\mu \d m_\lambda f,\quad \forall f\in L^2(\R).}
\end{equation}

To prove our first null-controllability result, as we mentioned before, by the Lebeau-Robbiano method, we only need to establish the following spectral inequality:
\begin{theorem}\label{spectral-inequality}
	Let $V$ be a potential satisfying Assumption~\ref{assump1} with $\beta_2\ge \beta_1>0$. Let $\Omega$ be thick of type $(\rho_s,\tau)$ with $\rho_s$ defined in~\eqref{rho}, $s\ge \beta_2 / 2$ and $\tau\ge 0$. Then there exists a constant $C=C(V,\Omega)>0$ depending only on $V$ and $\Omega$, such that for any $\lambda>0$ and any $\phi \in \mathcal{E}_\lambda(\sqrt{H})$, we have
 \begin{equation}
              \|\phi\|_{L^2(\R)}\le Ce^{C\lambda^{\zeta}}\|\phi\|_{L^2(\Omega)}
 \end{equation}
where $\zeta=\frac{4\tau }{\beta_1}+1$.
\end{theorem}

For the critical case $s=\frac{\beta_2}{2}$, we also obtain the following spectral inequality by replacing $\rho_s$ with $\rho$ given in~\eqref{new-rho}:
\begin{theorem}\label{spectral-inequality-2}
	Let $V$ be a potential satisfying Assumption~\ref{assump1} with $\beta_2\ge\beta_1>0$ . Let $\Omega$ be thick of type $(\rho,\tau)$ with $\rho$ defined in~\eqref{new-rho} for some $R>0$ and $\tau\ge 0$. Then for any $\varepsilon>0$, and any $R$ sufficiently small, there exists a constant $C=C(\varepsilon,R,V,\Omega)>0$ depending only on $\varepsilon$, $R$, $V$ and $\Omega$, such that for any $\lambda>0$ and any $\phi \in \mathcal{E}_\lambda(\sqrt{H})$, we have
 \begin{equation}
         \|\phi\|_{L^2(\R)}\le Ce^{C\lambda^{\zeta}}\|\phi\|_{L^2(\Omega)}
 \end{equation}
 where $\zeta=\frac{4\tau+\varepsilon}{\beta_1}+1$.
\end{theorem}

 For potentials under Assumption~\ref{assump2} which is slightly more restrictive than Assumption~\ref{assump1},  we obtain the following spectral inequality from thick sets with decaying density:
\begin{theorem}\label{spectral-inequality-3}
Let $V$ be a potential satisfying Assumption~\ref{assump2} with $\beta_2\ge \beta_1>0$, $\Omega$ be a thick set with decaying density which satisfies~\eqref{eqn-by-Ming} for some $L>0$, $\gamma\in [0,1)$ and $\tau\ge 0$. Then there exists a constant $C=C(V,\Omega)>0$ depending only on $V$ and $\Omega$, such that for any $\lambda>0$ and any $\phi\in\mathcal{E}_\lambda(\sqrt{H})$, we have
\begin{equation}
    \|\phi\|_{L^2(\R)}\le Ce^{C\lambda^\zeta\log(\lambda+1)}\|\phi\|_{L^2(\Omega)}
\end{equation}
where $\zeta=\frac{4\tau}{\beta_1}+\frac{\beta_2}{\beta_1}$.
\end{theorem}

\begin{remark}
The above new spectral inequalities improve the spectral inequalities in \cite[Theorem~1]{zhu2023spectral} and \cite[Theorem~1.1]{dicke2022spectral} from the sets that satisfy~\eqref{regular-thick-set} to thick sets with decaying density in the 1-dimensional case. There are two differences compared to those results:
\begin{enumerate}
    \item The methods used in \cite{zhu2023spectral,dicke2022spectral} are based on Carleman estimates, which are only restricted to the set $\Omega$ that contains the union of infinitely many equidistributed small open balls, see~\eqref{regular-thick-set}. 
        \item The exponent in Theorem~\ref{spectral-inequality-3} is ${\color{review2}\lambda^{4\tau/\beta_1+\beta_2/\beta_1}\log (\lambda+1)}$, while the one in \cite[Theorem~1]{zhu2023spectral} is $\lambda^{{2\tau}/{\beta_1}+{\beta_2}/{\beta_1}}$. {\color{review2} For a possible improvement of the exponent, there are two parts which are worse than the one in \cite{zhu2023spectral}. First, the extra $\log$ term arises from the application of Proposition~\ref{aux-spectral-inequality}, which is a spectral inequality for a more regular set $\widetilde{\Omega}$. This set includes a union of open intervals: there exists a sequence $\lbrace x_k\rbrace_{k\in\Z}\subset \R$ such that
        \begin{equation}
            \exists L>0,\sigma\ge 0:\quad\widetilde{\Omega}\cap (k+[0,L])\supset I_{\langle k\rangle^{-\sigma}L}(x_i),\quad \forall k\in\Z.
        \end{equation}
        We use Proposition~\ref{aux-spectral-inequality} for some $\sigma>0$, and this will give the $\log$ term.
        Hence the $\log$ term cannot be dropped unless we do not use this auxilary result. It is an ad-hoc method to reduce our problem to the regular one by the pigeonhole principle so that we can use Carleman estimates. Second, even without the $\log$ term, it is still worse than the one in \cite{zhu2023spectral}. It relys on the estimates of the constants in Theorem~\ref{pos-rmk-simplified}. It is possible to give a sharper exponent if: (1) one can improve the estimate~\eqref{estimate-alpha} in Theorem~\ref{pos-rmk-simplified} by replacing $\bigl|\log|\omega|\bigr|^2$ with $\bigl|\log|\omega|\bigr|$; (2) one assumes extra regularity of potentials and use a totally different propagation of smallness property, see \textit{e.g.}, \cite[Theorem~1]{zhu2024spectral}. }
\end{enumerate}
\end{remark}

\begin{remark}
    {\color{review2}    Compared to Theorem~\ref{spectral-inequality} and Theorem~\ref{spectral-inequality-2}, in Theorem~\ref{spectral-inequality-3} we extend the set $\Omega$ from thick of type $(\rho_s,\tau)$ to an $\Omega$ that is thick of type $(\rho_0,\tau)$, giving $I_{L\rho(x)}=[x-L.x+L]$. We believe that it is possible to extend this result further to the case where $\rho(x)$ grows with $|x|\to \infty$. This has been done in \cite{alphonse2022quantitative,martin2023spectral} for certain analytic potentials. For a general potential, it is not known whether we can extend it further under Assumption~\ref{assump2}. Noticing that an extra regularity condition in Assumption~\ref{assump2} (compared to Assumption~\ref{assump1}) is needed when we extend to $\rho_0$ case, therefore such an extension may needs some improvements of the regularity posed on the potential so that we have some sort of smoothing effects like those in \cite{alphonse2022quantitative,martin2023spectral}.}
\end{remark}

In \cite{su2023quantitative}, the authors considered the observability inequality of the 1D Schrödinger equation for $V(x) \in  L^\infty(\R)$, and proved it for $\Omega$ being $\gamma$-thick. 
To do so, they established the spectral inequality for the Schrödinger operator $H=-\partial_x^2+V(x)$ with $V(x)\in L^{\infty}(\R)$, we present it here as a comparison:
\begin{theorem}[{\cite[Lemma~3.8]{su2023quantitative}}]\label{cont-spectral-inequality}
    Let $\Omega$ be a $(1,\gamma)$-thick set and let $V\in L^{\infty}(\R)$. Then there exists a constant $C=C(V,\gamma)>0$ depending only on $V$ and $\gamma$, such that for any $\lambda>0$ and any $f\in L^2(\R)$, we have
    \begin{equation}
        \|\Pi_\lambda f\|_{L^2(\R)}\le Ce^{C\lambda}\|\Pi_\lambda f\|_{L^2(\Omega)}.
    \end{equation}
\end{theorem}
Though the forms of two spectral inequalities are the same, there are several key differences due to the lack of boundedness of potentials, see Subsection~\ref{sec-prf} for more details. 

To obtain the observability inequalities (Theorem~\ref{spectral-inequality} and Theorem~\ref{spectral-inequality-2}), and then Theorem~\ref{main} and Theorem~\ref{main-2}, we only need the following reformulated version of Lebeau-Robbiano method:
\begin{theorem}[{\cite[Theorem 2.8]{nakic2020sharp}}]\label{spectral-to-obs}
Let $P$ be a non-negative selfadjoint operator on $L^2\left( \R^{d} \right)$ and
$\omega \subset \R^{d}$ be measurable. Suppose that there are $\alpha_0\geq 1$, $\alpha_1\ge 0$ and $0<\zeta <2$ such that, for all $\lambda \ge 0$ and $\phi \in \mathcal{E}_\lambda(\sqrt{P})$,
\begin{equation}\label{2.1c}
\|\phi\|^2_{L^2(\R^{d})}\le \alpha_0 e^{\alpha_1\lambda^{\zeta}}\|\phi\|^2_{L^2(\Omega)}.
\end{equation}
Then there exist positive constants $\kappa_1,\kappa_2,\kappa_3>0$ depending only on $d$ and $\zeta$, such that for all $T>0$ and $g \in L^2(\R^{d})$ we have the observability estimate
\[
\|{\color{review1}e^{-PT}}g\|^2_{L^2(\R^{d})}\le \frac{C_{\mathrm{obs}}}{T}\int_0^{T}\|e^{-tP}g\|^2_{L^2(\omega)}\d t,
\] 
where the positive constant $C_{\mathrm{obs}}>0$ is given by
\begin{equation}\label{2.2c}
C_{\mathrm{obs}}=\kappa_1 \alpha_0^{\kappa_2}
\exp\left( \kappa_3 \alpha_1^{\frac{2}{2-\zeta}}T^{-\frac{\zeta}{2-\zeta}}  \right).
\end{equation}
\end{theorem}

Now with the new spectral inequalities and the above theorem, our observability inequalities are direct consequences of our spectral inequalities. Indeed, In Theorem~\ref{main}, to obtain our result by Theorem~\ref{spectral-to-obs}, it is sufficient to keep $\zeta<2$, with Theorem~\ref{spectral-inequality} the latter is equivalent to $ \tau<{\beta_1}/{4}$. The same approach can be used to obtain Theorem~\ref{main-2} and Theorem~\ref{main-3}.

\subsection{Previous results}\label{previous-results}
There have been numerous studies on the null-controllability of heat equations over the past decades, both on bounded domains with various boundary conditions and on unbounded domains. As our results deal with the heat equation on $\R$, we will focus on the Euclidean case, or some results that are highly related to the propagation of smallness. We may refer \textit{e.g.}, to  \cite{lebeau1995controle,lebeau1998null, alessandrini2008null, le2012carleman, apraiz2013null, lu2013lower, apraiz2014observability,escauriaza2017analyticity,burq2022propagation} for the results on bounded domains, and \cite{rouveyrol2024spectral} for the spectral inequality on the hyperbolic half-plane. {\color{review1}Throughout the article, the use of $\R^d$ means that it works for any $d$-dimensional Euclidean space.}

Consider the Schrödinger operator $H=-\Delta+V(x)$ defined on $\R^d$. First, we discuss the null-controllability and the spectral inequalities from thick sets (possibly with decaying density) for different kinds of potentials.

For the free case $V=0$, it was proved in~\cite{egidi2018sharp} and independently in~\cite{wang2019observable} that the heat equation is null-controllable from any thick set. The spectral inequality in this case is a reformulation of Kovrijkine's sharp version of the Logvinenko-Serreda Uncertainty Principle \cite{kovrijkine2001some}.  Furthermore, when the heat equation is equipped with a bounded, real-valued, suitably analytic potential $V$ vanishing at infinity, it was shown in \cite{lebeau2019spectral} that the spectral inequality still holds from thick sets. The spectral inequality is also established for the free heat equation on compact manifolds from measurable sets \cite{burq2022propagation} and on some noncompact manifolds from $\gamma$-thick sets \cite{burq2021propagation}. The null-controllability for general parabolic equations from $\gamma$-thick sets is also obtained in \cite{duan2021quantitative}.

For the harmonic case $V=|x|^2$, the spectral inequality (and hence the null-controllability of the heat equation) from different kinds of $\Omega$ was proved in \cite{beauchard2018null,beauchard2021spectral,egidi2021abstract,martin2023spectral,dicke2023uncertainty}. All of them are based on the explicit form of the Hermite functions which are the eigenfunctions of the harmonic oscillator. In particular, we mention that \cite{dicke2023uncertainty} allows control sets $\Omega\subset \R^d$ satisfying
\begin{equation}
    |\Omega\cap (k+(-\rho_s /2, \rho_s /2)^d)|\ge \gamma^{1+|k|^\sigma}\rho_s^d,\quad \forall k \in \Z^d
\end{equation}
with some parameters $\rho_s>0$, $\gamma\in(0,1]$, and $\alpha\in [0,1)$. 
For the case $V=|x|^{2k},k> 1$, the spectral inequality from $\gamma$-thick control sets was proved in \cite{alphonse2022quantitative} based on quantitative Agmon estimates.

All the above results required the real analyticity or even the explicit form of potentials. There are also some results for more general or less regular potentials. For instance, Dicke, Seelmann and
Veselić \cite{dicke2022spectral} recently considered the Schrödinger operator with power growth potentials: $V\in H^1_{\mathrm{loc}}(\R^d)$ that satisfies Assumption~\ref{assump2} in any $d$-dimension. They obtained the spectral inequality from a set
$\Omega$ that is more regular,  \textit{i.e.}, there exists an \textit{equidistributed} sequence of points $\{z_k\}_{k\in\Z^d}$ such that 
\begin{equation}
    \Omega\cap (k+(-1,1)^d)\supset B_{\gamma^{1+|k|^\sigma}}(z_k),\label{regular-thick-set}
\end{equation}
where $B_r(z)$ denotes the ball with radius $r$ and center $z$. The set $\Omega$ that satisfies the above condition is called \textit{$(\gamma,\sigma)$-distributed} in \cite[Definition~1.5]{jaming2023null}. Shortly after, the index $\zeta$ in $\lambda^{\zeta}$ of the spectral inequality was improved by Zhu and Zhuge in \cite{zhu2023spectral}. These are based on Carleman estimates and cannot be applied to rough sets.

For a less regular potential, all strategies mentioned above seem not to work. However, the spectral inequality for the 1D Schrödinger operator with bounded potentials from thick sets was proved in \cite{su2023quantitative}. This relies on the new result about the propagation of smallness for elliptic equations in the plane, see~\cite{zhu2023remarks}. The proof is limited to 1D because, in higher dimensions, we only have a weaker quantitative propagation of smallness, see~\cite{logunov2018quantitative} for the propagation of smallness in higher dimensions. Inspired by their work, we generalize their results to power growth potentials. The new obstacles occur in unbounded potentials, which will be discussed in the next subsection. 

{\color{review2}
Unlike the case $V=0$, we have very limited knowledge about some kinds of necessary conditions on the observable sets for the Schrödinger operators power growth potentials. One possible way is to use the same strategy in the proof of the case $V=0$ in \cite{wang2019observable}. This needs the pointwise estimate of the heat kernel for the Schrödinger operator with the power growth potential, see \textit{e.g.}, an estimate in \cite{sikora1997diagonal}. Besides, we need to choose a suitable way to adapt the proof in \cite{wang2019observable}. We do not pursue this direction here.
}

\subsection{Latest results on arbitrary dimensions }

{\color{myself}
In the context of arbitrary $\R^d$ or compact manifolds, there are two very recent preprints \cite{zhu2024spectral,balc2024quantitative} that were posted on the arxiv shortly after this paper.
They contain results for both unbounded and bounded rough potentials that are closely related to this work:

In \cite{zhu2024spectral}, the author considered power growth potential $V\in L_{\mathrm{loc}}^{\infty}(\R^d)\cap \mathcal{C}^{0,1}_{\mathrm{loc}}(\R^d\backslash \mathcal{B}_R)$ for some $R>0$, and established Theorem~\ref{spectral-inequality-3} with an exponent as sharp as the one in \cite{zhu2023spectral}. He used a different way to construct an extra ghost dimension to absorb the potential into the principal part of the operator. The author also considered the bounded potential and established Theorem~\ref{cont-spectral-inequality} in this case as well as a new quantitative propagation of smallness for gradients of solutions to some elliptic equations (see \cite[Lemma~4]{zhu2024spectral}).
Nevertheless, this work requires extra regularity condition $V\in \mathcal{C}^{0,1}_{\mathrm{loc}}(\R^d\backslash \mathcal{B}_R)$ on power growth potentials to obtain a sharper exponent.

In \cite{balc2024quantitative}, the authors considered $V\in L^\infty(M)$ both  when $M$ is a compact manifold and when $M=\R^d$. They use an auxiliary function (see \cite[Lemma~3.1]{balc2024quantitative}) to absorb the potential and so reducing the operator to a divergence form one. Based on two propagation of smallness results (see \cite[Theorem~2.1, Theorem~5.1]{logunov2018quantitative}), they obtain new propagation of smallness properties for nondivergence form operators (see \cite[Theorem~2.1, Theorem~2.2]{balc2024quantitative}).

}
\subsection{Outline of the proof}\label{sec-prf}
As we mentioned before, the proofs of our main results boil down to the proofs of spectral inequalities in Subsection~\ref{subsection-spectral-inequality}.

To prove these spectral inequalities, we follow a strategy similar to the one in~\cite{burq2021propagation} (see also \cite{burq2022propagation, su2023quantitative}), that is, lift the function $\phi\in\mathcal{E}_\lambda(\sqrt{H})$ into 2D so that the new function $\Phi$ satisfies the 2D elliptic equation in nondivergence form
\begin{equation}
    -\Delta\Phi(z)+V(x)\Phi(z)=0.\label{ode-whose-sol}
\end{equation}
Then, like the approach in \cite[Section~2]{logunov2020landis} (see also \cite[Subsection~3.1]{su2023quantitative}), we construct an auxiliary ODE whose solution $\varphi$ allows us to rewrite \eqref{ode-whose-sol} as 
\begin{equation}
    \nabla\cdot\left(\varphi^2\nabla \left(\frac{\Phi(z)}{\varphi(x)}\right)\right)=0.
\end{equation}
It is a 2D elliptic equation in divergence form. Then we can apply a very recent work \cite{zhu2023remarks} on the propagation of smallness for elliptic equations in divergence form in the plane.
 It roughly says that, given a line segment $\ell_0\subset B_1$ with normal vector $\mathrm{e}_0$ and a measurable set $\omega\subset \ell_0$, for any $H_{\mathrm{loc}}^2$ solution $\Psi$ of the elliptic equation 
 \begin{equation}
    \nabla\cdot(A\nabla \Psi)=0 {\color{review1}\text{ in }B_4},\quad A\nabla\Psi\cdot \mathbf{e}_0=0 \text{ in } \ell_0,
\end{equation}
there exist constants $C$ and $\alpha$ such that the inequality
 \begin{equation}
     \sup_{B_1}{\Psi}\le C\sup_{\omega}|\Psi|^\alpha \sup_{B_4}|\Psi|^{1-\alpha},
 \end{equation}
 holds.
 
In the remaining of this subsection, we will talk in detail about the key steps that do not occur in the proof of Theorem~\ref{cont-spectral-inequality}, which shows the difficulties in our proofs, and the differencies between the Schrödinger operators with unbounded potentials and those with bounded potentials.

In the proof of Theorem~\ref{spectral-inequality}, to obtain the decay of sensor sets in the spectral inequalities (\textit{i.e.}, allowing $\tau>0$ in Theorem~\ref{spectral-inequality}), we need to revisit each step in the proof of propagation of smallness given in \cite{zhu2023remarks}, so that we can find the explicit forms of the constants $\alpha$ and $C$. Another important reason that we are allowed to gain the decay of sensor sets comes from the localization property of power growth potentials (see Lemma~\ref{localization}), \textit{i.e.},
\begin{equation}
    \|\phi\|_{L^2(\R)}\le 2\|\phi\|_{L^2(I_{\lambda})},\quad \forall \phi \in \mathcal{E}_\lambda(\sqrt{H}),
\end{equation}
where $I_\lambda=[-C\lambda^{2/\beta_1},C\lambda^{2/\beta_1}]$ and $C>0$ is a positive constant. By using this property, we only need to consider the distribution of control sets in $I_\lambda$.  

Let us first recall the proof in the bounded potential case in \cite{su2023quantitative}. The authors decomposed the spectral inequality into the local estimates on each interval $[n,n+1], n\in\Z$ and got the following estimate for any $\varepsilon>0$ by using 2D propagation of smallness in~\cite{zhu2023remarks}:
\begin{equation}\label{zz-1}
    \|\Phi\|^2_{L^2(D_{1,n})}\le \frac{C}{\varepsilon}\|\phi\|^2_{L^2(\omega)}+C\varepsilon^{\frac{\alpha}{1-\alpha}}\|\Phi\|^2_{L^\infty(D_{2,n})}
\end{equation}
where $\Phi$ is the lift of $\phi\in \mathcal{E}_\lambda(\sqrt{H})$, $D_{1,n}=[n,n+1]\times [-1/2,1/2]$, $D_{2,n}=[n-1,n+2]\times [-3/2,3/2]$, $\alpha\in(0,1)$ and $C>0$. Then they summed up over $n\in\Z$ and obtain
\begin{equation}
\sum_{n\in\Z}\|\Phi\|^2_{L^2(D_{1,n})}\le \frac{C}{\varepsilon}\sum_{n\in\Z}\|\phi\|^2_{L^2(\omega)}+C\varepsilon^{\frac{\alpha}{1-\alpha}}\sum_{n\in\Z}\|\Phi\|^2_{L^\infty(D_{2,n})}.  
\end{equation}
Finally they finished the proof by estimating each sum in the above inequality.
Thanks to the boundedness condition posed on the potential, they can treat the estimate on each interval uniformly. Under our assumption, one of the difficulties is that the potential is not bounded, so this can no longer be done. To overcome this, we try to change the method of decomposition, \textit{i.e.}, we want $\R=\bigcup_{n\in\Z}I_n$ with $|I_n|\to 0$ when $|n|\to \infty$, where the $I_n$'s are chosen so that, after a proper rescaling of the equation on each interval, there is still a uniform bound on the potential in each interval.

By the above approach, we obtain a uniform estimate~\eqref{zz-1} in each interval.
However, due to the decreasing property of $|I_n|$, the $H^2$-norm of the terms on right-hand side that appears later in the proof (see the proof of Lemma~\ref{lma-different})  grows to infinity when the intervals are away from the origin, which prevents us from merging each piece into a whole spectral inequality. Thanks to the localization property again, we only need to merge finitely many intervals that are located in the bounded interval $I_\lambda$, see the merged inequality~\eqref{a-8}. 

Another problem that occurs for unbounded potentials is the estimate of the second term on the right-hand side of~\eqref{a-8}. We prove the estimate~\eqref{a-7}, \textit{i.e.} ~Lemma~\ref{lma-different}, by using almost all good properties coming from the power growth of the potentials. As a comparison, in the case of bounded potentials, the estimate~\eqref{a-7} is trivial and just one line in~\cite{su2023quantitative}.

The proof of Theorem~\ref{spectral-inequality-2} is very similar to the proof of Theorem~\ref{spectral-inequality}. The only difference is that the rescaled potential restricted in each interval (which is slightly larger than those in the proof of Theorem~\ref{spectral-inequality}) is not uniformly bounded. This is allowed because we have the explicit dependence of $\alpha$ and $C$ on the ellipticity parameter $\Lambda$. This is another reason for which we need the exact form of the propagation of smallness.  However, due to the rapid decay of $\alpha$ (see~\eqref{estimate-alpha}) with respect to $\Lambda$, we cannot relax the lengths of intervals to be the same. Hence we cannot obtain the spectral inequality for thick sets by the method under Assumption~\ref{assump1}.

However, with a slightly more restrictive Assumption~\ref{assump2}, we can relax the lengths of intervals and obtain the spectral inequality from thick sets with decaying density, \textit{i.e.}, Theorem~\ref{spectral-inequality-3}. To prove it, we use an auxiliary spectral inequality (see Proposition~\ref{aux-spectral-inequality}) from a more regular set, which is established by Carleman estimates. With this at hand, combining it with the method in the proof of Theorem~\ref{spectral-inequality}, we obtain the desired result. Precisely speaking, for each measurable set $\Omega\cap[n,n+1]$, we split $[n,n+1]$ into several small intervals with the same length. Then we may find at least one interval $I_n'$ such that the measure of the control region in it is not less than the average. Now the problem of rapid decay of $\alpha$ with respect to $n$ is replaced by the polynomial decay of the intervals $I_n'$. Therefore it can be reduced to the proof of the spectral inequality from the union of $I_n'$, \textit{i.e.}, the proof of Proposition~\ref{aux-spectral-inequality}. The latter is the same as the proof in~\cite{zhu2023spectral}.

\subsection{Notation} 
We shall denote the constants by $C$, $C'$, $c$, $\kappa$, $d$ (possibly with some numerical subscripts), and use the notation $C(\alpha)$ to show that the constant $C(\alpha)$ depends on $\alpha$. The value of the constants may be different from line to line.
Sometimes we use $A\lesssim_{\alpha} B$ (resp. $A\gtrsim_{\alpha} B$) if there exists a positive constant $C(\alpha)$ such that $A\le C(\alpha) B$ (resp. $A\ge C(\alpha) B$), or simply $A\lesssim B$ (resp. $A\gtrsim B$) if the dependence is obvious. We use $A\asymp B$ if both $A\lesssim B$ and $B\lesssim A$ hold. We use the notation $a_n\sim b_n$ if we have $\lim_{n\to \infty}\frac{a_n}{b_n}=1$.  We also denote by $I_r(x)$ the interval with length $2r$ and center $x$. The symbol $\langle \cdot \rangle$ always denotes the Japanese bracket defined by $\langle x\rangle :=(1+|x|^2)^{1 /2}$ for any $x\in\R$.

\section{Reductions}\label{sec-2}
In this section, we reduce the assumptions posed on control sets and potentials for future convenience.

\subsection{Reduction of thick sets}\label{reduction-thick}
We reduce thick sets of type $(\rho_s,\tau)$ to an equivalent form, which is more convenient to be used to prove Theorem~\ref{spectral-inequality} but hides its geometric meaning.

Let  $L>0$ and $s\ge 0$, we define the sequence $\left\{x_n\right\}_{n\in \N}$ of real positive numbers as the follows: set $x_0=0,x_1=L$, define $x_{n}$ for each $n\in \left\{1,2,\cdots\right\} $ by using the recurrence formula
	\begin{equation}\label{recurrence-formula-origin}
		x_{n+1}=x_n+L\left( \frac{1}{x_n} \right) ^{s}, \quad \forall n \in \N_+=\left\{1,2,\cdots\right\}.
	\end{equation}
Then we define
\begin{equation}
	I_0:=[-L,L],\, I_{n}:=\left[ x_n,x_{n+1} \right] ,\quad \forall n\in \N_{+}
\end{equation}
and
\begin{equation}
	I_{n}:=-I_{|n|},\quad \forall  n\in -\N_+.
\end{equation}

We have the following standard approximation estimate for the sequence:
\begin{lemma}\label{growing-estimate-lemma}
	Let ${\color{review1}\left\{x_n\right\} _{n\in \N_+}}$ be the sequence given by~\eqref{recurrence-formula-origin} for some $L>0$, then we have
	\begin{equation}
  \lim_{n\to \infty}\frac{x_n}{\left((s+1)Ln\right)^{\frac{1}{s+1}}}=1.
	\end{equation}
\end{lemma}
\begin{proof}
	First of all, we obtain immediately from~\eqref{recurrence-formula-origin} that the sequence ${\color{review1}\left\{x_n\right\}_{n\in\N_+}}$ is increasing and $\lim_{n\to\infty}x_n=\infty$. {\color{review1}Indeed, if $x_n\le C$ for some positive constant $C>0$, then
 \begin{equation}
     x_{n+m}-x_{n}=\sum_{k=0}^{m-1}\bigl(x_{n+k}-x_{n+k}\bigr)\ge m\frac{1}{C^s}\to \infty \text{ when } m\to \infty,
 \end{equation}
 which makes a contradiction.}
 What we need to prove is the relation $x_n^{s+1}\sim (s+1)Ln$. Indeed we have
	\begin{equation}
	    \begin{aligned}
	        x^{s+1}_n\sim& x_n^{s+1}-x_0^{s+1}=\sum_{k=1}^{n}\left(x_k^{s+1}-x_{k-1}^{s+1}\right)=\sum_{k=1}^n \left(\left(x_k+\frac{L}{x_k^s}\right)^{s+1}-x_k^{s+1}\right)\\
	       \sim & \sum_{k=1}^n x_k^{s+1}\left(\left(1+\frac{L}{x_k^{s+1}}\right)^{s+1}-1\right)\sim \sum_{k=1}^n (s+1)L\sim (s+1)Ln,   
	    \end{aligned}
	\end{equation}
	and this finishes the proof.
\end{proof}

The following lemma gives an alternative description of thick sets of type $(\rho_s,\tau)$ defined in Definition~\ref{def-by-prof}.
\begin{lemma}\label{equiv-lma}
    A measurable set $\Omega$ is thick of type $(\rho_s,\tau)$ if and only if there exist $\gamma>0$ and $L>0$ such that, with $\lbrace x_n\rbrace_{n\in\N}$ satisfying~\eqref{recurrence-formula-origin}
    \begin{equation}
		|\Omega\cap I_n|\ge \gamma^{\langle x_n\rangle^\tau} |I_n|,\quad \forall n\in\Z.\label{thick} 
    \end{equation}
\end{lemma}

\begin{proof}
First, we assume that the set $\Omega$ is thick of type $(\rho_s,\tau)$ given by Definition~\ref{def-by-prof}. This means that there exist positive constants $L'$ and $\gamma'>0$ such that the density condition
\begin{equation}\label{equiv-assump-1}
    |\Omega\cap I_{L'\rho_s(x)}(x)|\ge {\gamma'}^{\langle x\rangle ^\tau}|I_{L'\rho_s(x)}(x)|
\end{equation}
holds for all $x\in\R$. Now we choose $L=\max\left\{1, 4L'\right\}$ and construct the sequence by the recurrence formula~\eqref{recurrence-formula-origin}. Due to the symmetrical construction of the intervals, we only need to show that there exists a positive constant $\gamma>0$ such that
~\eqref{thick} hold
for all integers $n\in\N$.

We define 
\begin{equation}
    {\color{review1}c_0=\frac{L}{2} \text{ and } c_{n}:=\frac{x_n+x_{n+1}}{2}=x_n+\frac{L}{2}x_n^{-s},\quad \forall n\in\N_+.\label{z-1}}
\end{equation}
Then we obtain from~\eqref{z-1} that 
 \begin{equation}
     c_n-L'\langle c_n \rangle^{-s}> x_n+ 2L'|x_n|^{-s} -L'\langle x_n \rangle^{-s}> x_n+L'\langle x_n\rangle ^{-s}>x_n
 \end{equation}
 and
 \begin{equation}
     c_n+L'\langle c_n\rangle^{-s}<c_n+L'x_n^{-s}\le x_n+\frac{L}{2}x_n^{-s}+\frac{L}{4}x_n^{-s}<x_{n+1}
 \end{equation}
 for all $n\in\N$.
 These two inequalities are equivalent to
 \begin{equation}
     x_n<c_n-L'\langle c_n\rangle^{-s}<c_n+L'\langle c_n\rangle^{-s}<x_n+x_n^{-s}=x_{n+1},
 \end{equation}
 which implies $I_{n}\supset I_{L'\rho_s(c_{n})}(c_{n})$. Using this inclusion relation and the assumption~\eqref{equiv-assump-1}, we obtain
\begin{equation}\label{e-2}
    |\Omega\cap I_{n}|\ge |\Omega\cap I_{L'\rho_s(c_{n})}(c_{n})| \ge {\gamma'}^{\langle c_{n}\rangle^\tau}|\Omega\cap I_{\rho_s(c_{n})}(c_{n})|= {\color{review1}\frac{2L'{\gamma'}^{\langle c_{n}\rangle ^\tau}}{\langle c_{n}\rangle^s}}.
\end{equation}
On the other hand, by Lemma~\ref{growing-estimate-lemma}, it is easy to see that  $x_n\sim c_n$, \textit{i.e.}, there exist positive constants $C_1$ and $C_2$ depending only on $L$ and $s$, such that
\begin{equation}
    C_1 \langle c_n\rangle \le x_n\le C_2 \langle c_{n}\rangle.
\end{equation}
Then
\begin{equation}
    |I_n|=Lx_n^{-s}\le {\color{review1}L}C_1^{-s}\langle c_n\rangle^{-s}.\label{e-3}
\end{equation}

Combining~\eqref{e-2} and~\eqref{e-3} we obtain
\begin{equation}
    |\Omega\cap I_n|\ge 2{\color{review1}L^{-1}L'}C_1^{s} {\gamma'}^{\langle c_n\rangle^\tau}|I_n|.
\end{equation}
The inequality~\eqref{thick} is obtained from the above by choosing $\gamma$ small enough so that $\gamma< \inf_{n\in \N} (2{\color{review1}L^{-1}L'}C_1^s)^{1/\langle c_n\rangle^\tau}\gamma'$, where the right hand side is positive since $\langle c_n\rangle \ge 1 $ for all $n\in\N$.

Now we assume that the converse is true for some $\gamma>0$ and $L>0$. Then we claim that there exists $\gamma'\in(0,\gamma)$ and $L'>L$ which will be chosen later, such that we have the density condition~\eqref{equiv-assump-1} for all $x\in\R$.
Indeed, notice that for each $x\in \R$ and the corresponding interval $I_{L'\rho_s(x)}(x)$, we have some $k\in \Z$ such that $x\in I_k$. Without loss of generality, we assume $k\ge 0$. 
We first show that we can choose $L'$ large enough such that $x\in I_k\subset I_{L'\rho_s (x)}(x)$ for all $k\ge 0$.  If $k=0$, then
\begin{equation}
    x-L'\langle x\rangle ^{-s}\le L-L'\langle L\rangle ^{-s} <-L
\end{equation}
and
\begin{equation}
    x+L'\langle x\rangle^{-s}\ge -L+ L'\langle L\rangle^{-s} >L
\end{equation}
when $L'>2L\langle L\rangle ^s$. 
If $k>0$, then
\begin{equation}
    x-L'\langle x\rangle^{-s}\le x_{k+1}-L'\langle x_{k+1}\rangle^{-s}< x_k
\end{equation}
\begin{equation}
    x+L'\langle x\rangle ^{-s}\ge x_{k}+L'\langle x_{k+1}\rangle ^{-s} > x_{k+1}
\end{equation}
when $L'>\left(\frac{\langle x_{k+1}\rangle }{|x_k|}\right)^{s}$. The above four inequalities imply that $x\in I_k\subset I_{L'\rho_s (x)}(x)$ for each $k\in \N$ if we choose $L'$ which satisfies
\begin{equation}\label{f-choice}
{\color{review1}L'> \max\left\{ 2L\langle L\rangle ^s, \left(\frac{\langle x_1 \rangle}{x_0}\right)^s \right\}\ge \max_{k\in\N_+}\left\{ 2L\langle L\rangle^s, \left(\frac{\langle x_{k+1}\rangle }{x_k}\right)^{s} \right\},}
\end{equation}
since $\left\{\frac{\langle x_{k+1}\rangle}{x_k}\right\}_{k\in\N}$ is decreasing by the recurrence formula~\eqref{recurrence-formula-origin}.
Given the choice of $L'$, we have 
\begin{equation}\label{f-lowerbound}
    \frac{|I_k|}{|I_{L'\rho_s(x)}(x)|}=\frac{Lx_n^{-s}}{2L'\langle x\rangle^{-s}}\ge \frac{Lx_n^{-s}}{2L'\langle x_n\rangle ^{-s}}\ge \frac{L}{2L'}
\end{equation}
for $k\ge 1$
and 
\begin{equation}\label{f-lowerbound-2}
    \frac{I_0}{I_{L'\rho_s(x)}(x)}\ge \frac{L}{L'}>\frac{L}{2L'}.
\end{equation}
Then for all integers $k\in\N$, we obtain 
\begin{equation}
    |\Omega\cap I_{L'\rho_s(x)}(x)|\ge |\Omega\cap I_k|\ge \gamma^{\langle x_k\rangle ^\tau}|I_k|\ge \gamma^{\langle x_k\rangle^\tau}\frac{L}{2L'}|I_{L'\rho_s(x)}(x)|,
\end{equation}
where we use the inclusion relation $I_{L'\rho_s(x)}(x)\supset I_k$ and the lower bound in~\eqref{f-lowerbound} and~\eqref{f-lowerbound-2}. Then we finish the proof of~\eqref{equiv-assump-1} by choosing $\gamma'>0$ such that it satisfies
\begin{equation}
    \gamma'< \gamma\inf_{k\in\N}\left(\frac{L}{2L'}\right)^{\frac{1}{\langle x_k\rangle^\tau}}=\gamma,
\end{equation}
since $\inf_{k\in\N}\langle x_k\rangle= 1$ and $2L'>L$. 
\end{proof}

\subsection{Reduction of the scale}

Without loss of generality, we can always assume that $L=1$ and $\gamma\in (0,\frac{1}{2})$. Indeed, we may do the linear transform $x=L^{-1}y$ and choose possibly different $\gamma$ in Definition~\ref{def-by-prof} and $c_1,c_2,c_3$ in Assumption~\ref{assump1}, so that the new parameter $L=1$. We can always choose $\gamma<\frac{1}{2}$ since a smaller one still preserves the density condition~\eqref{thick}. 

Based on the above reductions, the proofs of our main results under Assumption~\ref{assump1} can be reduced to the case that $\Omega$ satisfies  $\gamma\in(0,\frac{1}{2})$ with $L=1$. For the case under Assumption~\ref{assump2}, we may do the same reduction. 

\subsection{Reduction of potentials}
Let $u(t)$ be a solution of~\eqref{heat} and $\theta>0$, then $U(t):=e^{-\theta t}u$ is a solution of
\begin{equation}
	\partial_t U-\partial_x^2U+V(x)U+\theta U={\color{review1}e^{\theta t}h(t,x)\un_{\Omega}},\quad U\lvert_{t=0}=u_0 \in L^2(\R).
\end{equation}
Through this transformation, we can always assume 
\begin{equation}
    1\le c_1\langle x\rangle^{\beta_1}\le V(x)\le c_2\langle x\rangle^{\beta_2}
\end{equation}
in Assumption~\ref{assump1} and Assumption~\ref{assump2}.

\section{Exact propagation of smallness}

Consider the 2D elliptic equation in divergence form
\begin{equation}
    \nabla \cdot (A(z)\nabla\Psi(z))=0,\quad z\in B_4.\label{pos}
\end{equation} 
Here $z=(x,y)\in\R^2$, $A(z)=(a_{jk}(z))_{2\times 2}$ is a real symmetric matrix with $a_{jk}(z)\in L^\infty(B_4)$, $j,k=1,2$, and satisfies the so-called ellipticity condition: there exists a constant $\Lambda>1$ such that
\begin{equation}\label{ellipticity}
    \Lambda^{-1}\le \xi^\tau A(z)\xi \le \Lambda,\quad \forall \xi \in B_1,\,z\in B_4.
\end{equation}

In this section, we aim to prove the following quantitative propagation of smallness for solutions of~\eqref{pos}:

\begin{theorem}\label{pos-rmk-simplified}
Let $A=(a_{jk})_{2\times 2}$ be a real symmetric matrix with bounded measurable entries satisfying~\eqref{ellipticity}.
Let $\omega\subset B_1\cap \ell_0$ with $|\omega|\in(0,\frac{1}{2})$ for some line $\ell_0$ in \,$\R^2$ with the normal vector $\textbf{e}_0$. There exist some constants $\alpha=\alpha(\Lambda,|\omega|)\in (0,1)$ and $C=C(\Lambda,|\omega|)>0$ depending only on $\Lambda$ and $|\omega|$, such that for any real-valued $H^2_{\mathrm{loc}}$ solution $\Psi$ of~\eqref{pos} with $A\nabla\Psi \cdot \textbf{e}_0=0$ on $B_1\cap \ell_0$, we have
\begin{equation}
    \sup_{B_1}|\Psi|\le C\left(\sup_\omega |\Psi|^\alpha\right)\left(\sup_{B_2}|\Psi|^{1-\alpha}\right),
\end{equation}
where the constants $C$ and $\alpha$ satisfy 
\begin{equation}\label{estimate-alpha}
    \frac{e^{-d_2 \Lambda^2}}{\bigl|\log| \omega|\bigr|^2}\le \alpha \le \frac{e^{-d_1\Lambda^2}}{\bigl|\log|\omega|\bigr|^2}
\end{equation}
and
\begin{equation}\label{estimate-c}
    e^{d_1 \Lambda^2}   \le C\le e^{d_2 \Lambda^2},
\end{equation}
for some positive constants $d_1$ and $d_2$ which depend on the line $\ell_0$.
\end{theorem}

\begin{remark}
Zhu in his recent work \cite{zhu2023remarks}, first proved this in a more general case where $\omega\subset B_1\cap\ell_0$ is only needed to have a positive $\delta$-Hausdorff content for arbitrary $\delta>0$. It is worthwhile to mention that for the case $\omega$ being open, the result has been obtained earlier in \cite[Proposition~1]{alessandrini2008null}.  The result in \cite{zhu2023remarks} did not give the exact constants $\alpha$ and $C$, which are important here to obtain the decay distribution of the control sets. However, following the proof in \cite{zhu2023remarks} very carefully, we get the explicit constants for the particular case that $\omega$ has a positive 1-dimensional Lebesgue measure on the line, that is, the relations~\eqref{estimate-alpha} and~\eqref{estimate-c}. 
\end{remark}

Before giving the proof, we will need several ingredients here about quasiconformal mappings and holomorphic functions that we recall now.

\subsection{Preliminaries}
We first recall basic definitions and properties of quasiregular and quasiconformal mappings. All of this can be found in~\cite{astala2008elliptic} or~\cite{ahlfors2006lectures}.

Let $f$ be a complex function, we write the differential of $f$ in terms of the Wirtinger derivative as
\begin{equation}
    Df(z) h =\frac{\partial f}{\partial z}(z)h+\frac{\partial f}{\partial \overline{z}}(z)\overline{h}, \quad \text{for } h\in\C.
\end{equation}
The norm of the derivative and Jacobian can be written as
\begin{equation}
    |Df(z)|=\left|\frac{\partial f}{\partial z}(z)\right|+\left|\frac{\partial f}{\partial \overline{z} }(z)\right|\text{ and }
    Jf(z)=\left|\frac{\partial f}{\partial z}(z)\right|^2-\left|\frac{\partial f}{\partial \overline{z}}(z)\right|^2.
\end{equation}

\begin{definition}
Let $U$ and $V$ be connected open subsets of $\C$, $f:U\to V$ and $f\in W_{\mathrm{loc}}^{1,2}(U)$, we call the following inequality
        \begin{equation}
            |Df(z)|^2\le KJf(z)\label{distortion-inequality}
        \end{equation}
       with $K\ge 1$ the \textit{distortion inequality}.
    \begin{enumerate}
        \item The function $f$ is a \textit{$K$-quasiregular mapping} if it is an orientation-preserving mapping and the distortion inequality~\eqref{distortion-inequality} holds for almost every $z\in U$;
        \item The function $f$ is a \textit{$K$-quasiconformal mapping} if it is an orientation-preserving homeomorphism from $U$ to $V$ and the distortion inequality~\eqref{distortion-inequality} holds for almost every $z\in U$.
    \end{enumerate}
\end{definition}

It is well-known that quasiconformal mappings are Hölder continuous. Precisely speaking, we have the following famous theorem given in \cite{mori1957quasi}:
\begin{theorem}[Mori's Theorem]\label{mori}
Let $K\ge 1$ and $f(z)$ be a $K$-quasiconformal mapping from $B_r$ onto itself, normalized by $f(0)=0$. Then for any $r>0$, $z_1,z_2\in B_r$ and $z_1\neq z_2$, we have
    \begin{equation}
        r^{1-K}\left(\frac{1}{16}\right)^K|z_1-z_2|^K\le |f(z_1)-f(z_2)|\le 16r^{1-\frac{1}{K}}|z_1-z_2|^{\frac{1}{K}}.\label{mori-rescale}
    \end{equation}
\end{theorem}
Mori's Theorem is usually stated for $r=1$. The above general form is directly obtained by the scaling $g(z)= r^{-1}f(rz)$.

Now we gather some results about holomorphic functions.

The first lemma is Jensen's formula (see, \textit{e.g.},\cite[Theorem~15.18]{rudin1987real}):
\begin{lemma}[Jensen's formula]
Let $0<r_1<r_2<\infty$. Let $f$ be a holomorphic function in $B_{r_2}$ with $f(0)\neq 0$ and $a_1,a_2,\cdots,a_N$ are the zeros of $f$ in $B_{r_1}$ (repeated according to their respective multiplicity), then we have
    \begin{equation}
        \log |f(0)|=\sum_{n=1}^N\log \left( \frac{|a_n|}{r_1} \right)+\frac{1}{2\pi }\int_0^{2\pi }|f(r_1e^{i\theta})|\d \theta.
    \end{equation}
\end{lemma}
Jensen's formula is often used to estimate the number of zeros of holomorphic functions in an open disk.
\begin{corollary}\label{corollary-number-of-zeros}
    Let $0<r_1<r_2<r_3<\infty$ and $a\in\C$. Let $f$ be a holomorphic function in $B_{r_3}(a)$ with $f(a)\neq 0$. Let $N$ be the number of zeros of $f$ in $B_{r_1}(a)$ (counted with multiplicity), then we have 
    \begin{equation}
        N\le \frac{\log M-\log |f(a)|}{\log r_2-\log r_1}, 
    \end{equation}
    where $\displaystyle M:=\max_{|z|=r_2}|f(z)|$.
\end{corollary}
The second lemma is Hadamard's three-circle theorem (see, \textit{e.g.}, \cite[Page 264]{rudin1987real}):
\begin{lemma}\label{lma-3-circle-thm}
    Let $f$ be a holomorphic function in $B_R$ for $0<R<\infty$. Let $M(r)$ be the maximum of $|f(z)|$ on the circle $|z|=r$ for $0<r<R$. Then we have
    \begin{equation}
        \log \left(\frac{r_3}{r_1}\right) \log M(r_2)\le \log \left(\frac{r_3}{r_2}\right)\log M(r_1)+\log \left(\frac{r_2}{r_1}\right)\log M(r_3)
    \end{equation}
    for any three concentric circles of radii $0<r_1<r_2<r_3<R$.
\end{lemma}

For $\delta>0$, we denote by $\mathcal{H}_\delta$ the \textit{$\delta$-dimensional Hausdorff content}, that is, for a subset $E\subset\R^2$, we define
\begin{equation}
    \mathcal{H}_{\delta}(E)=\inf \left\{ \sum_{n=1}^\infty r_n^\delta:E\subset \bigcup_{n=1}^\infty B_{r_n}(a_n),\, a_n\in\R^2 \right\}.
\end{equation}
Let $\omega\subset B_1\cap \ell_0$ satisfy $|\omega|>0$ for some line $\ell_0$ in $\R^2$. From the definition of 1D Lebesgue measure and $\delta$-Hausdorff content, we have
\begin{equation}\label{hausdorff-relation}
    |\omega|=\inf \left\{\sum_{n=1}^{\infty}|I_n|:\omega\subset \bigcup_{n=1}^\infty I_n,\, I_n\subset \ell_0\right\}=2\mathcal{H}_1(\omega).
\end{equation}

The third well-known lemma illustrates the relation between Hausdorff contents and Hölder continuous mappings, we recall the proof to show the exact constant:
\begin{lemma}\label{hausdorff-quasiregular-lemma}
Let $G(z)$ be a Hölder continuous mapping, that is, there exist positive constants $C$ and $\kappa\in (0,1)$ such that
\begin{equation}
    |G(z_1)-G(z_2)|\le C|z_1-z_2|^\kappa.
\end{equation}
Let $E\subset B_1$.
Then we have
\begin{equation}
    \mathcal{H}_{\delta}\left(G(E)\right)\le C^\delta\mathcal{H}_{\kappa\delta}\left(E\right).
\end{equation}
\end{lemma}
\begin{proof}
Given any $\varepsilon>0$ and any sequence of balls $\left\{ B_{r_i}(z_i)\right\}_{i\in\mathcal{I}}$ covering $E$ such that $\sum_{i\in\mathcal{I}}{\color{review1}r_i}^{\kappa\delta}\le \mathcal{H}_{\kappa\delta}(E)+\varepsilon$, then $\left\{G\left(B_{r_i}(z_i)\right)\right\}_{i\in\mathcal{I}}$ is a covering of $G(E)$. By the Hölder continuity of $G$, we obtain $\left\{B_{r_i'}(G(z_i))\right\}$ the covering of $G(E)$ where $r_i':={\color{review1}C r_i^\kappa}$ for all $i\in\mathcal{I}$.
Then we obtain
\begin{equation}
    \mathcal{H}_\delta \left(G(E)\right)\le \sum_{i\in\mathcal{I}} r_i'^\delta=\sum_{i\in\mathcal{I}} C^\delta  {\color{review1}r_i}^{\kappa\delta}\le C^\delta\mathcal{H}_{\kappa\delta}(E)+\varepsilon.
\end{equation}
As $\varepsilon$ is arbitrary, the result follows.
\end{proof}

The fourth lemma is the Remez-type inequality for polynomials, which can be found in~\cite[Theorem~4.3]{Friedland2017}:
\begin{lemma}\label{ramez}
    Let $P(z)$ be a polynomial of degree $N$. Let $E\subset B_1$. Then for any $\delta>0$, we have
    \begin{equation}
        \sup_{B_1}|P(z)|\le \left(\frac{6e}{\mathcal{H}_\delta (E)}\right)^{\frac{N}{\delta}}\sup_{E}|P(z)|.
    \end{equation}
    Here $\mathcal{H}_\delta(E)$ denotes the $\delta$-dimensional Hausdorff content of $E$.
\end{lemma}

Last, we give the standard interior estimates of derivatives of harmonic functions (see, \textit{e.g.}, \cite[Theorem~2.10]{gilbarg1977elliptic}):
\begin{lemma}\label{interior-grd-estimate}
    Let $u$ be harmonic in an open connected set $\Omega$ in $\R^2$ and let $\Omega'$ be any compact subset of $\Omega$. Then for any multi-index $\alpha$ we have
    \begin{equation}
        \sup_{\Omega'}|D^\alpha u|\le \left(\frac{2|\alpha|}{\mathrm{dist}(\Omega',\partial\Omega)}\right)^{|\alpha|}\sup_{\Omega}|u|,
    \end{equation}
    where $\alpha=(\alpha_1,\alpha_2)\in \N\times \N$, $D^\alpha=\partial_x^{\alpha_1}\partial_y^{\alpha_2}$ and $|\alpha|=\alpha_1+\alpha_2$.
\end{lemma}

\subsection{Proof of Theorem~\ref{pos-rmk-simplified}}
The proof follows the one given in~\cite{zhu2023remarks} and a more detailed version in~\cite[Appendix~A]{su2023quantitative}. Indeed, we choose each radius and calculate each constant explicitly in the proof given in \cite[Appendix~A]{su2023quantitative}.

We first construct a quasiregular mapping. We denote
\begin{equation}
    S:=\begin{pmatrix}
    0 & -1\\
    1 & 0
    \end{pmatrix}:\R^2\to \R^2 \text{ and } S^2=-\mathrm{Id}.
\end{equation}
For any solution $\phi$ of~\eqref{pos}, the field $(SA\nabla \phi)$ is curl-free, and thus from Poincaré's Lemma, there exists a Sobolev function $\psi\in W^{1,2}_{\mathrm{loc}}(B_4)$, such that
\begin{equation}
    \nabla\psi = SA\nabla\phi=\begin{pmatrix}
    0 & -1\\
    1 & 0
    \end{pmatrix}A\nabla \phi.
\end{equation}
The function $\psi$ is called the $A$-harmonic conjugate of $\phi$ and is unique up to an additive constant. 
For any $a\in B_1$, we construct
\begin{equation}
    f_a=\phi+i\psi_a,
\end{equation}
where $\psi_a$ is the $A$-harmonic conjugate of $\phi$ with $\psi_a(a)=0$. The new function $f_a$ is a $\Lambda+\Lambda^{-1}$-quasiregular mapping. It then follows that
$f_a=F\circ G$ on $B_2$, where $F$ is holomorphic in $B_2$ and $G$ is a $\Lambda+\Lambda^{-1}$-quasiconformal homeomorphism from $B_2$ onto $B_2$ which verifies $G(0)=0$. Moreover, from~\cite[Corollary~5.9.2]{astala2008elliptic} we know that $G$ has a continuous homeomorphic extension to the boundary $\partial B_2$. For all $(z_1,z_2)\in B_2\times B_2$, by choosing $r=2$ in Theorem~\ref{mori} we have
\begin{equation}\label{holder-inequality}
    \left(\frac{1}{32}\right)^K|z_1-z_2|^K\le |G(z_1)-G(z_2)|\le 32 |z_1-z_2|^\frac{1}{K},
\end{equation}
where $K=\Lambda+\Lambda^{-1}\ge 1$.

For any $a\in B_1$, we consider the holomorphic homeomorphism of $B_2$ given by
\begin{equation}
    R_a(z)=4\frac{z-G(a)}{4-\overline{G(a)}z}:B_2\to B_2.
\end{equation}
Then $f_a$ can be rewritten as
\begin{equation}
    f_a=(F\circ R_a^{-1})\circ (R_a\circ G):=F_a\circ G_a\quad \text{ on } B_2,
\end{equation}
where $F_a=F\circ R_a^{-1}$ is holomorphic in $B_2$, and $G_a=R_a\circ G$ is a $(\Lambda +\Lambda^{-1})$-quasiconformal homeomorphism from $B_2$ onto $B_2$ which verifies $G_a(a)=0$. The derivative of $R_a(z)$ is
\begin{equation}
    D_zR_a(z)=4\frac{4-|G(a)|^2}{(4-\overline{G(a)}z)^2}.
\end{equation}
Combining this with $r=2-\left(\frac{1}{32}\right)^K$ implies
\begin{equation}
    \max_{w\in B_r}|D_zR_a(\omega)|\le \frac{16}{(4-2r)^2}\le 32^{1+2K},
\end{equation}
so that for all $(z_1,z_2)\in B_1\times B_1$, we have
\begin{equation}
    |G_a(z_2)-G_a(z_1)|\le 32^{1+2K} |G(z_2)-G(z_1)|.\label{deform}
\end{equation}
We obtain from~\eqref{holder-inequality}, that for all $(z_1,z_2)\in B_1\times B_1$
\begin{equation}
\left(\frac{1}{32}\right)^{(2+2K)K}|z_1-z_2|^K\le |G_a(z_2)-G_a(z_1)|\le 32^{2+2K} |z_1-z_2|^{\frac{1}{K}},\label{new-holder}
\end{equation}
where the left part is obtained by using the same argument for the $\Lambda+\Lambda^{-1}$-quasiconformal function $G_a^{-1}$.
In particular, we choose $r$ as
\begin{equation}\label{new-r}
    r=2-\left(\frac{1}{32}\right)^{2(1+K)K}\Rightarrow G_a(B_1)\subset B_r.
\end{equation}

\medskip
We first establish the following proposition:
\begin{proposition}\label{j-prop}
    Let $\omega\subset B_1\cap \ell_0$ satisfy $|\omega|>0$ for {\color{review1}the line $\ell_0=\lbrace (x,y)\in\R^2:y=0\rbrace$}. Then for some $z_0\in \omega$, there exist some constants $\alpha=\alpha(\Lambda,|\omega|)\in (0,1)$ and $C(\Lambda,|\omega|)>0$, such that for any real-valued $H^2_{\mathrm{loc}}$ solution $\phi$ of~\eqref{pos} with its $A$-harmonic conjugate satisfying $\psi_{z_0}(z_0)=0$, we have
    \begin{equation}
        \sup_{B_1}|\phi|\le C\left(\sup_{\omega}|\phi+i\psi_{z_0}|^\alpha\right)\left(\sup_{B_2}|\phi|^{1-\alpha}\right),
    \end{equation}
    where the constants $C$ and $\alpha$ satisfy~\eqref{estimate-alpha} and~\eqref{estimate-c} for some positive numerical constants $d_1$ and $d_2$.
\end{proposition}
\begin{proof}
{\color{review1}The strategy of the proof is to decompose the function $f_{z_0}=\phi+i\psi_{z_0}$ into $f_{z_0}=F_{z_0}\circ G_{z_0}$ for some $z_0\in B_1$, and reduce it to the holomorphic case since $F$ is a holomorphic function, that is, reduce to the inequality
\begin{equation}
    \sup_{G_{z_0}(B_1)}|F_{z_0}|\le C\left(\sup_{G_{z_0}(\omega)}|F_{z_0}|^{\alpha}\right)\left(\sup_{B_2}|\phi|^{1-\alpha}\right)
\end{equation}
for some $C>0$ and $\alpha\in (0,1)$. To achieve this, we need to make sure that the quasiconformal mapping $G_{z_0}$ will not change the ``shape'' of the domain too much, which is controlled by Mori's Theorem. 

First, we need to choose a ball $B_{r_1}\subset B_1$ with $r_1<1$ so that the ``shape'' of $G_a(B_1)$ for any $a\in B_{r_1}$ can be uniformly controlled and $\omega$ restricted to this ball has positive $1$-dimensional Lebesgue measure. Otherwise, if we treat any $a\in B_1$ directly, we cannot have a uniform estimate for the ``shape'' of $G_a(B_1)$ through Mori's Theorem.} This kind of $r_1$ can be chosen as $r_1=1-a|\omega|$ for any $a\in (0,\frac{1}{2})$.
For simplicity, we choose
\begin{equation}\label{choose-r-1}
    r_1=1-\frac{|\omega|}{8},
\end{equation}
so that we have
\begin{equation}
    |\omega\cap B_{r_1}|\ge \frac{3}{4}|\omega|.
\end{equation}
We consider $\displaystyle z\in \partial B_1$ and $a\in B_{r_1}$. By~\eqref{new-holder} we have
    \begin{equation}
        |G_a(z)|=|G_a(z)-G_a(a)|\ge \left(\frac{1}{32}\right)^{2(1+K)K}|z-a|^{{\color{review1}K}}\ge \left(\frac{1}{32}\right)^{2(1+K)K}{\color{review1}\left(\frac{|\omega|}{8}\right)^K}.
    \end{equation}
    Take $r_2$ as
    \begin{equation}\label{choose-r-2}
        6r_2=\left(\frac{1}{32}\right)^{2(1+K)K}{\color{review1}\left(\frac{|\omega|}{8}\right)^K}.
    \end{equation}
    Then we have
    \begin{equation}\label{choose-of-r-2}
        B_{6r_2}\subset \bigcap_{a\in B_{r_1}}G_a(B_1).
    \end{equation}
    Now we claim the following:
\begin{claim}
There exist $r_3\in (0,1)$ and $z_0\in \omega\cap B_{r_1'}$ with $r_1'=1-\frac{|\omega|}{4}<r_1$ such that
    \begin{equation}
        |\omega\cap B_{r_3}(z_0)|\ge \left(\frac{1}{32}\right)^{{\color{review1}4}(1+K)^3}|\omega|^{{\color{review1}K^2}+1}\label{r-3-condition-1}
    \end{equation}
    and 
    \begin{equation}
        G_{z_0}\left(B_{r_3}(z_0)\right)\subset B_{r_2}.\label{r-3-condition-2}
    \end{equation}
\end{claim}

\begin{proof}[Proof of the claim]
Since for any $z\in \partial B_{r_3}(z_0)$ and any $z_0\in \omega\cap B_{r_1'}$, we have
\begin{equation}
    |G_{z_0}(z)|=|G_{z_0}(z)-G_{z_0}(z_0)|\le 32^{2(1+K)}|z-z_0|^{\frac{1}{K}}\le 32^{2(1+K)}r_3^{\frac{1}{K}},
\end{equation}
to satisfy~\eqref{r-3-condition-2}, it is enough to choose
\begin{equation}\label{choose-r-3}
    r_3=\left(\frac{1}{32}\right)^{2(1+K)K}r_2^K={\color{review1}\left(\frac{1}{6^K8^{K^2}}\right)\left(\frac{1}{32}\right)^{2(1+K)^2K}|\omega|^{K^2}.}
\end{equation}
Notice that $r_3\ll r_1$ and $r_3\ll r_1-r_1'$, we may divide the interval $\ell_0\cap B_{r_1'}$ into $\displaystyle \lceil\frac{2r_1}{r_3}\rceil$ intervals $\ell_{n},n=1,2\cdots,\lceil \frac{2r_1}{r_3}\rceil,$ of equal length $\displaystyle \frac{|\ell_0\cap B_{r_1'}|}{\lceil \frac{2r_1}{r_3}\rceil}$, where $\lceil a\rceil$ denotes the largest integer smaller than $a+1$. By the pigeonhole principle, then there exists at least one piece $\ell_{n_0}$ such that 
\begin{equation}
    |\omega\cap \ell_{n_0}|\ge \frac{|\omega|}{2\lceil \frac{2r_1}{r_3} \rceil},
\end{equation}
{\color{review1}where we used the fact $|\omega\cap B_{r_1'}|\ge |\omega|-2(1-r_1')=\frac{|\omega|}{2}$.} Choose one $z_0\in \ell_{n_0}$ and then we have
\begin{equation}\label{aux-z-0}
    |\omega\cap B_{r_3}(z_0)|\ge \frac{|\omega|}{2\lceil \frac{2r_1}{r_3} \rceil}. 
\end{equation}
By~\eqref{choose-r-1} and~\eqref{choose-r-3} we have
\begin{equation}
    2\lceil \frac{2r_1}{r_3} \rceil\le 4\times 48^{{\color{review1}K^2}}\times 32^{2(1+K)^2K} \,\frac{2-\frac{|\omega|}{4}}{|\omega|^{\color{review1}K^2}}\le {32^{{\color{review1}4}(1+K)^3}}|\omega|^{-{\color{review1}K^2}}
\end{equation}
{\color{review1}where we used the facts $6^K8^{K^2}<48^{K^2}$, $2-\frac{|\omega|}{4}\le 2$ and then $ 8\times 48^{K^2}\times 32^{2(1+K)^2K}\le 32^{4(1+K)^3}$.}
Combining this with~\eqref{aux-z-0} we obtain the desired~\eqref{r-3-condition-1}, proving the claim.
\end{proof}

Now we go back to the proof of Proposition~\ref{j-prop} and denote $\widetilde{\omega}:=\omega\cap B_{r_3}(z_0)$.

Let $a_1,a_2,\cdots,a_N$ be the zeros of $F_{z_0}$ in $B_{2r_2}$ (repeated according to their respective multiplicity). Let $a_0\in\partial B_{r_2}$ be such that $|F_{z_0}(a_0)|$ is the maximum of $|F_{z_0}(z)|$ on $B_{r_2}$. We consider the following complex polynomial $P_{z_0}(z)$ and holomorphic function $h_{z_0}(z)$ which is nonvanishing on $B_{2r_2}$,
\begin{equation}
    P_{z_0}=\prod_{n=1}^N(z-a_n) \text{ and }h_{z_0}(z)=\frac{F_{z_0}(z)}{P_{z_0}(z)}.
\end{equation}
On the one hand, we denote by $\widetilde{N}$ the number of zeros of $F_{z_0}$ in $B_{3r_2}(a_0)$. From $B_{2r_2}\subset B_{3r_2}(a_0)$, $B_{\frac{7}{2}r_2}(a_0)\subset B_{5r_2}$ and Corollary~\ref{corollary-number-of-zeros}, we have
\begin{equation}
\begin{aligned}\label{number-zeros}
      N\le \widetilde{N}\le &\frac{1}{\log \frac{7}{6}}\left(\log\sup_{B_{\frac{7}{2}r_2}(a_0)}|F_{z_0}|-\log |F_{z_0}(a_0)|\right)\\  
      &\le \frac{1}{\log\frac{7}{6}}\left(\log \sup_{B_{5r_2}}|F_{z_0}|-\log\sup_{B_{r_2}}|F_{z_0}|\right).
\end{aligned}
\end{equation}
Then, using the definition of $h_{z_0}$ and $P_{z_0}$, we see that
\begin{equation}\label{aux-before}
    \sup_{B_{r_2}}|F_{z_0}|=|F_{z_0}(a_0)|\le |h_{z_0}(a_0)||P_{z_0}(a_0)|\le (3r_2)^N|h_{z_0}(a_0)|.
\end{equation}
Using the maximum modulus principle on $B_{5r_2}$, we find
\begin{equation}\label{maximum-modulus}
    \sup_{B_{5r_2}}|h_{z_0}|\le \left(\sup_{\partial B_{5r_2}}|F_{z_0}|\right) \bigg /\left(\inf_{\partial B_{5r_2}}|P_{z_0}|\right)\le \left(\frac{1}{3r_2}\right)^N\sup_{B_{5r_2}}|F_{z_0}|.
\end{equation}
Since $h_{z_0}$ is a holomorphic non-vanishing function, $\log|h_{z_0}|$ is a harmonic function on $B_{2r_2}$. From Harnack's inequality, we have
\begin{equation}\label{harnack}
    \sup_{B_{r_2}}\left(\sup_{B_{5r_2}}\log|h_{z_0}|-\log |h_{z_0}(z)|\right)\le 9\inf_{B_{r_2}}\left(\sup_{B_{5r_2}}\log |h_{z_0}|-\log|h_{z_0}(z)| \right),
\end{equation}
{\color{review1}where the constant comes from $\left(\frac{R+r}{R-r}\right)^2=9$ by choosing $R=2r_2$ and $r=r_2$.}
It follows that 
\begin{equation}\label{aux-harnack}
    |h_{z_0}(a_0)|^9\sup_{B_{5r_2}}|h_{z_0}|\le \left(\sup_{B_{5r_2}}|h_{z_0}|^9 \right)\left(\inf_{B_{r_2}}|h_{z_0}|\right).
\end{equation}
Combining~\eqref{aux-before},~\eqref{maximum-modulus} and~\eqref{aux-harnack}, we obtain
\begin{equation}\label{aux-harnack-2}
    \left(\sup_{B_{r_2}}|F_{z_0}|^9\right)\left(\sup_{B_{5r_2}}|h_{z_0}|\right)\le\left(\sup_{B_{5r_2}}|F_{z_0}|^9\right) \left(\inf_{B_{r_2}}|h_{z_0}|\right).
\end{equation}
On the other hand, using Lemma~\ref{ramez} and {\color{review1}the inclusion condition~\eqref{r-3-condition-2}}, we find
\begin{equation}\label{ramez-2}
    \sup_{B_{r_2}}|P_{z_0}|\le \left(\frac{6e}{\mathcal{H}_\alpha(G_{z_0}(\widetilde{\omega}))}\right)^{\frac{N}{\alpha}}\sup_{G_{z_0}(\widetilde{\omega})}|P_{z_0}|.
\end{equation}

Now we estimate the $\frac{1}{K}$-dimensional Hausdorff content of $G_{z_0}(\widetilde{\omega})$. Notice that $G_{z_0}$ satisfies~\eqref{new-holder}, we may use Lemma~\ref{hausdorff-quasiregular-lemma} with $G=G^{-1}_{z_0}$,  $\kappa=\frac{1}{K}$, $C=1024^{1+K}$ and $\delta=1$ to obtain
\begin{equation}\label{hausdorff-quasi-change}
    \mathcal{H}_{\frac{1}{K}}\left(G_{z_0}(\widetilde{\omega})\right)\ge \frac{2}{32^{2(1+K)}} |\widetilde{\omega}|.
\end{equation}
Substituting~\eqref{hausdorff-quasi-change} into~\eqref{ramez-2}, we obtain
\begin{equation}
    \sup_{B_{r_2}}|P_{z_0}|\le \left(\frac{3e \times 32^{2(1+K)}}{|\widetilde{\omega}|}\right)^{KN}\sup_{G_{z_0}(\widetilde{\omega})}|P_{z_0}|.
\end{equation}
Recall that we defined $\widetilde{\omega}=\omega\cap B_{r_3}(z_0)$, and substituting~\eqref{r-3-condition-1} into the above inequality we obtain
\begin{equation}
\begin{aligned}
    \sup_{B_{r_2}}|P_{z_0}|&\le \left(\frac{3e\times 32^{{\color{review1}4}(1+K)^3+2(1+K)}}{|\omega|^{{\color{review1}K^2}+1}}\right)^{KN}\sup_{G_{z_0}(\widetilde{\omega})}|P_{z_0}|\\
    &\le \left(\frac{32^{{\color{review1}4}(1+K)^4}}{|\omega|^{{\color{review1}K^2}+1}}\right)^{KN}\sup_{G_{z_0}(\widetilde{\omega})}|P_{z_0}|.    
\end{aligned}
\end{equation}
Combining this with~\eqref{aux-harnack-2}, we obtain
\begin{equation}
    \begin{aligned}
        \sup_{B_{r_2}}|F_{z_0}|^{1+9}\le & \left(\sup_{B_{r_2}}|F_{z_0}|^9\right)\left(\sup_{B_{5r_2}}|h_{z_0}|\right)\left(\sup_{B_{r_2}}|P_{z_0}|\right)\\
        \le & \left(\frac{32^{4(1+K)^4K}}{|\omega|^{(K^2+1)K}}\right)^N\left(\sup_{B_{5r_2}}|F_{z_0}|^9\right)\left(\inf_{B_{r_2}}|h_{z_0}|\right)\left(\sup_{G_{z_0}(\widetilde{\omega})}|P_{z_0}|\right)\\
        \le & \left(\frac{32^{4(1+K)^4K}}{|\omega|^{(K^2+1)K}}\right)^N\left(\sup_{B_{5r_2}}|F_{z_0}|^9\right)\left(\sup_{G_{z_0}(\widetilde{\omega})}|F_{z_0}|\right)
    \end{aligned}
\end{equation}
Defining 
\begin{equation}
    C_0:=\frac{32^{4(1+K)^4K}}{|\omega|^{(K^2+1)K}}
\end{equation}
and combining the above inequality with~\eqref{number-zeros}, we obtain
\begin{equation}
    \sup_{B_{r_2}}|F_{z_0}|^{10+\frac{\log C_0}{\log (7/6)}}\le \left(\sup_{B_{5r_2}}|F_{z_0}|^{9+\frac{\log C_0}{\log (7/6)}}\right)\left(\sup_{G_{z_0}(\widetilde{\omega})}|F_{z_0}|\right).
\end{equation}
Thus we have
\begin{equation}\label{aux-3-ball}
    \sup_{B_{r_2}}|F_{z_0}|\le \left(\sup_{B_{5r_2}}|F_{z_0}|^{1-\alpha_1}\right)\left(\sup_{G_{z_0}(\widetilde{\omega})}|F_{z_0}|^{\alpha_1}\right),
\end{equation}
where
\begin{equation}
\begin{aligned}
        \alpha_1:=&\frac{1}{10+\frac{\log C_0}{\log (7/6)}}=\frac{\log (7/6)}{10\log(7/6)+4(1+K)^4K\log 32-(K^2+1)K\log |\omega|}\\
        \asymp & \frac{1}{K^3(K^2+\bigl|\log|\omega|\bigr|)}.\label{alpha-1}
\end{aligned}
\end{equation}
Recall the definitions of $r$ in~\eqref{new-r} and $r_2$ in~\eqref{choose-of-r-2}, then by Hadamard three-circle theorem we obtain
\begin{equation}\label{hadamard-3-ball}
    \begin{aligned}
        \sup_{G_{z_0}(B_1)}|F_{z_0}|\le \sup_{B_r}|F_{z_0}|\le &\left(\sup_{B_{r_2}}|F_{z_0}|^{\alpha_2}\right)\left(\sup_{B_{\frac{2+r}{2}}}|F_{z_0}|^{1-\alpha_2}\right),
    \end{aligned}
\end{equation}
where
\begin{equation}\label{alpha-2}
    \alpha_2=\frac{\log \frac{2+r}{2r}}{\log \frac{2+r}{2r_2}}
    =
    {\log\frac{4-\left(\frac{1}{32}\right)^{2(1+K)K}}{4-2\left(\frac{1}{32}\right)^{2(1+K)K}}}\bigg/{\log \frac{4-\left(\frac{1}{32}\right)^{2(1+K)K}}{\frac{1}{3}\left(\frac{|\omega|}{8}\right)^K\left(\frac{1}{32}\right)^{2(1+K)K}}}\asymp \frac{32^{-2(1+K)K}}{K^2+K\bigl|\log|\omega|\bigr|}.
\end{equation}
Combining~\eqref{aux-3-ball} with~\eqref{hadamard-3-ball} and using $B_{5r_2}\subset B_{\frac{2+r}{2}}$, we obtain
\begin{equation}\label{aux-last}
    \sup_{G_{z_0}(B_1)}|F_{z_0}|\le \left(\sup_{G_{z_0}(\widetilde{\omega})}|F_{z_0}|^{\alpha_1\alpha_2}\right)\left(\sup_{B_{\frac{2+r}{2}}}|F_{z_0}|^{1-\alpha_1\alpha_2}\right).
\end{equation}
From the definition of $F_{z_0}$, the holomorphic function $F_{z_0}$ can be written as
\begin{equation}\label{expansion}
   F_{z_0}=f_{z_0}\circ G_{z_0}^{-1}=\phi\circ G_{z_0}^{-1}+i\psi_{z_0}\circ G_{z_0}^{-1}. 
\end{equation} 
Therefore, combining with $\psi_{z_0}\circ G_{z_0}^{-1}(0)=0$ and the Cauchy-Riemann equation, we have
\begin{equation}
    \begin{aligned}
            \psi_{z_0}\circ G_{z_0}^{-1}(x,y)&=\int_0^1(x,y)\cdot \nabla (\psi\circ G_{z_0}^{-1})(x\sigma,y\sigma)\d \sigma\\
            &=\int_0^1(-y,x)\cdot \nabla (\phi\circ G_{z_0}^{-1})(x\sigma,y\sigma)\d \sigma\\
            &\le  \sup_{B_{\frac{2+r}{2}}}|\nabla(\phi\circ G_{z_0}^{-1})|.
    \end{aligned}
\end{equation}
Then from Lemma~\ref{interior-grd-estimate}, we obtain
\begin{equation}\label{interior-eqn}
\begin{aligned}
      \sup_{B_{\frac{2+r}{2}}}|\psi_{z_0}\circ G_{z_0}^{-1}|&\le \sup_{B_{\frac{2+r}{2}}}|\nabla (\phi\circ G_{z_0}^{-1})|\le \frac{8}{2-r}\sup_{\frac{6+r}{4}}|\phi\circ G_{z_0}^{-1}|\\
      &\le 8\times 32^{2(1+K)K}\sup_{\frac{6+r}{4}}|\phi\circ G_{z_0}^{-1}|.  
\end{aligned}
\end{equation}
Hence, using~\eqref{expansion} again and the above inequality, we find
\begin{equation}
    \begin{aligned}
        \sup_{B_{\frac{2+r}{2}}}|F_{z_0}|&\le \sup_{B_{\frac{2+r}{2}}}|\phi\circ G^{-1}_{z_0}|+\sup_{B_{\frac{2+r}{2}}}|\psi_{z_0}\circ G_{z_0}^{-1}|\\
        &\le 32^{2(1+K)^2}\sup_{B_{\frac{6+r}{4}}}|\phi\circ G_{z_0}^{-1}|\le 32^{2(1+K)^2}\sup_{B_2}|\phi|.
    \end{aligned}
\end{equation}
{\color{review1}Combining the above inequality with~\eqref{aux-last} and noticing that $\widetilde{\omega}\subset \omega$}, we obtain
\begin{equation}
    \sup_{G_{z_0}(B_1)}|F_{z_0}|\le 32^{2(1+K)^2(1-\alpha_1\alpha_2)}\left(\sup_{G_{z_0}(\omega)}|F_{z_0}|^{\alpha_1\alpha_2}\right)\left(\sup_{B_2}|\phi|^{1-\alpha_1\alpha_2}\right).
\end{equation}
    Taking $C=32^{2(1+K)^2(1-\alpha)}$ and $\alpha=\alpha_1\alpha_2$. Remember that $\bigl|\log|\omega|\bigr|\ge \log 2$, $\alpha_1$ and $\alpha_2$ have been given in~\eqref{alpha-1} and~\eqref{alpha-2} respectively, it is easy to check that $\alpha<\frac{1}{2}$ and two constants satisfy~\eqref{estimate-alpha} and~\eqref{estimate-c}. Note that the estimate $\alpha<\frac{1}{2}$ here is used to give a lower bound of $C$ in~\eqref{estimate-c}. Therefore, the proof is completed.
\end{proof}

Now we go back to the proof of Theorem~\ref{pos-rmk-simplified}:
\begin{proof}[Proof of Theorem~\ref{pos-rmk-simplified}]
{\color{review1}Without loss of generality, we choose the line to be the standard case $\ell_0=\lbrace (x,y)\in \R^2: y=0 \rbrace$. Indeed, for arbitrary line, we can choose a smaller ball $B_{r_0}$ which has the center located in the line and change the coordinate system to make it be horizontal. By using Hadamard's three-circle theorem for elliptic equation with the smallest ball being $B_{r_0}$, the left of the proof is the same as the standard case.} Let $z_0\in\omega$ be the point chosen in Proposition~\ref{j-prop} and $\psi_{z_0}$ be the $A$-harmonic conjuegate of $\Psi$ satisfying $\psi_{z_0}(z_0)=0$. By $A\nabla \Psi_{z_0}\cdot \textbf{e}_0=0$ and the definition of $\psi_{z_0}$, we obtain
\begin{equation}
    \nabla\psi_{z_0}\cdot \textbf{e}^{\perp}_0 = \left(SA\nabla \Psi\right)\cdot\textbf{e}_0^{\perp}=A\nabla\Psi\cdot \left(-S\textbf{e}_0^{\perp}\right)=A\nabla \Psi \cdot \textbf{e}_0=0. 
\end{equation}
Since $\psi_{z_0}(z_0)=0$, we obtain $\psi_{z_0}(z)=0$ on $\ell_0$. Hence from $\omega\subset \ell_0$ and Proposition~\ref{j-prop}, we obtain
\begin{equation}
    \sup_{B_1}|\Psi|\le C\left(\sup_{\omega}|\Psi+i\psi_{z_0}|^\alpha\right)\left(\sup_{B_2}|\Psi|^{1-\alpha}\right)\le C\left(\sup_{\omega}|\Psi|^\alpha\right)\left(\sup_{B_2}|\Psi|^{1-\alpha}\right).
\end{equation}
This finishes the proof.
\end{proof}

\section{Proof of spectral inequalities}
In this section, we first introduce the localization property of eigenfunctions. Then we give the propagation of smallness for elliptic equations in nondivergence form. Last, we finish the proof of spectral inequalities, \textit{i.e.}, Theorem~\ref{spectral-inequality}, Theorem~\ref{spectral-inequality-2} and Theorem~\ref{spectral-inequality-3}.

\subsection{Properties of eigenfunctions}
In this subsection, we recall some basic properties of the Schrödinger operator.

Let $V$ be a real-valued nonnegative function on  $\R^d$ with $V\in L^{\infty}_{\mathrm{loc}}(\R^d)$ such that
\begin{equation}
	\lim_{|x| \to \infty} V(x)=\infty
\end{equation}
and there exists some integer $m$,
\begin{equation}
    \lim_{|x|\to \infty}x^{m}V(x)=0.
\end{equation}
and consider the associated Schrödinger operator
\begin{equation}
	Hf(x):=H_{V}f(x):=-\Delta f(x)+V(x)f(x).
\end{equation}
Then this operator is well-defined on $\mathcal{S}(\R^d)$ and can be extended to a unique (unbounded) self-adjoint operator on $L^2(\R^d)$. As we mentioned before, it has positive discrete spectrum $\{\lambda_k\}_{k\in\N}$ with $\lambda_k\to \infty$ when $k\to\infty$ (see \cite[Section~2.3]{berezin1991schrodinger}). Now we present the localization property which is vital for the proof of the spectral inequality:
\begin{lemma}\label{localization}
	Assume that $V\in L^{\infty}_{\mathrm{loc}}(\R^d)$ satisfies $V(x)\ge c|x|^{\beta}$ with $c,\beta>0$. Then there exists a constant $C:=C(c,\beta,d)$ depending only on $c$, $\beta$ and $d$, such that for all $\lambda>0$ and $\phi \in \mathcal{E}_\lambda(\sqrt{H})$, we have
	\begin{equation}
		\|\phi\|_{L^2(\R^d)}\le 2\|\phi\|_{L^2\left(B_{C\lambda^{2 /\beta}}(0)\right)}
	\end{equation}
	and
	\begin{equation}
	    \|\phi\|_{H^1(\R^d)}\le 2\|\phi\|_{H^1\left(B_{C\lambda^{2 /\beta}}(0)\right)}.
	\end{equation}

\end{lemma}
The above lemma can be found in \cite[Proposition~3]{zhu2023spectral}. The same results can also be found in \cite[Theorem~1.4]{dicke2022spectral} and \cite[Lemma~2.1]{gagelman2012spectral} with an extra mild condition for $DV$. In the case $d=1$, the two inequalities in Lemma~\ref{localization} can be written as
\begin{equation}
    \|\phi\|_{L^2\left(\R\right)}\le 2\|\phi\|_{L^2\left(I_\lambda\right)}
\end{equation}
and
\begin{equation}
    \|\phi\|_{H^1(\R)}\le 2\|\phi\|_{H^1\left(I_\lambda\right)},
\end{equation}
	where 
	\begin{equation}\label{bounded-interval}
	    I_{\lambda}:=[-C\lambda^{2 /\beta},C\lambda^{2 /\beta}]
	\end{equation}
	and $C=C(c,\beta)$. 

Besides, we also have the following local estimates for the gradients of eigenfunctions which is vital for us:
\begin{lemma}\label{grd-estimate}
Assume that $V\in L^\infty_{\mathrm{loc}}(\R^d)$ satisfies $V(x)\ge c|x|^\beta$ with $c,\beta>0$. Let $\lbrace \phi_k\rbrace_{k\in\N}$ be the eigenfunctions with increasing eigenvalues $\lbrace \lambda_k^2\rbrace_{k\in\N}$ of the Schrödinger operator $H=-\Delta+V$. Then for every $r>0$ and $z\in \R^d$, we have
\begin{equation}\label{eqn-grd-estimate}
    \|\nabla \phi_k\|^2_{L^2(B_{r}(z))}\le \left(1+\frac{8}{r^2}\right)(1+\lambda_k^2)\|\phi_k\|^2_{L^2(B_{2r}(z))}.
\end{equation}
\end{lemma}

It is a local Caccioppoli
inequality and the proof can be found in \cite[Lemma~2.8]{jaming2023null}. In the case $d=1$, the inequality in Lemma~\ref{grd-estimate} can be written as
\begin{equation}
\|D\phi_k\|^2_{L^2(I_{r}(x))}\le \left(1+\frac{8}{r^2}\right)(1+\lambda_k^2)\|\phi\|^2_{L^2(I_{2r}(x))}\label{caccio-1}
\end{equation}
for all $r>0$ and $x\in\R$.

\subsection{Propagation of smallness: nondivergence form}\label{subsec-4-2}
In this subsection, we introduce the propagation of smallness for solutions of 2D elliptic equations. For future reference, we first give some notations which we frequently use.

    Remember from Subsection~\ref{reduction-thick} that  $\lbrace x_n\rbrace_{n\in\N}$ is defined by~\eqref{recurrence-formula-origin}, $I_n=[x_n,x_{n+1}]$ for all $n\in \N_+$ and $I_0=[-x_0,x_0],\,x_0=1$, $I_{-n}=-I_n$ for all $n\in\N$. We define
\begin{equation}
	{\color{review1}I_{1,n}:= I_n,\,I_{2,n}:=\left[ x_{n}-|I_n|,x_{n+1}+|I_{n}| \right] , \,
	I_{3,n}:=\left[ x_n-2|I_n|,x_{n+1}+2|I_n| \right] ,\quad \forall n\in \N.}
\end{equation}
We also denote by {\color{review1}$I_{j,n}:=-I_{j,|n|}$} for all $n\in-\N_{+}$ and $j=1,2,3$.
Define
\begin{equation}
	D_{j,n}:=I_{j,n}\times \left[ -\frac{(2j-1)|I_{n}|}{2},\frac{(2j-1)|I_{n}|}{2} \right], \quad  \forall n\in\N, j=1,2,3,
\end{equation}
and
\begin{equation}
	D_j:=[1-j,j]\times \left[ -\frac{(2j-1)}{2},\frac{(2j-1)}{2} \right],\quad \forall n\in\N, j=1,2,3.
\end{equation}

First, we introduce a local $L^\infty$-norm estimate for the second-order ODE
\begin{equation}
    -\varphi''(x)+V(x)\varphi(x)=0.\label{ode}
\end{equation}
{\color{myself}We have the following lemma which is a particular case of the elliptic theory and can be found in \cite[Lemma~3.1]{balc2024quantitative}:
\begin{lemma}\label{lma-ode}
    Let $I=[a,b]$ be a finite interval and let $V\in {\color{review1}L^\infty}(I)$ with $V\ge 0$. Then there exists a positive solution $\varphi\in H^2([a,b])$ of~\eqref{ode} on the interval $I$ such that
    \begin{equation}
        1\le\varphi\le e^{(b-a)\|V\|_{L^\infty(I)}^{\frac{1}{2}}},\quad \forall x\in I.\label{new-ode-sol}
    \end{equation}
\end{lemma}
\begin{proof}
    We first solve a second-order ODE problem with a boundary condition, by~\cite[Theorem~9.15]{gilbarg1977elliptic}, there exists $\varphi \in H^2([a,b])$ such that
    \begin{equation}
        -\varphi''(x)+V(x)\varphi(x)=0,\quad x\in [a,b],\,\varphi(a)=\varphi(b)= e^{(b-a)\|V\|_{L^\infty(I)}^{\frac{1}{2}}}.
    \end{equation}
    Define
    \begin{equation}
        \phi_{-}(x)=e^{(b-x)\|V\|_{L^\infty(I)}^{\frac{1}{2}}} \text{ and } \phi_+(x)=e^{(b-a)\|V\|_{L^\infty(I)}^{\frac{1}{2}}},\quad \forall x\in [a,b].
    \end{equation}
    Then we have
    \begin{equation}
        -\phi_{-}''+V\phi_{-}\le 0,\quad \forall x\in (a,b), \quad\phi_{-}(a),\phi_{-}(b)\le e^{(b-a)\|V\|_{L^\infty(I)}^{\frac{1}{2}}},
    \end{equation}
    and
    \begin{equation}
        -\phi_{+}''+V\phi_{+}\ge 0,\quad \forall x\in (a,b),  \quad\phi_{+}(a),\phi_{+}(b)\ge e^{(b-a)\|V\|_{L^\infty(I)}^{\frac{1}{2}}}.
    \end{equation}
    So by the maximal principle stated in \cite[Theorem~8.1]{gilbarg1977elliptic}, we have
    \begin{equation}
        \phi_{-}\le \phi(x)\le \phi_{+},\quad \forall x\in (a,b).
    \end{equation}
    This implies~\eqref{new-ode-sol} and finishes the proof.
\end{proof}
}

Next, we introduce the $L^2$-propagation of smallness
for $H^2_{\mathrm{loc}}$ solution of the following 2D elliptic equation in nondivergence form
\begin{equation}
	-\Delta \Phi(z)+V(x)\Phi(z)=0\, \text{ with }\,  {\color{review1}\partial_y \Phi\lvert_{y=0}=0}.\label{2d-elliptic}
\end{equation}
Following the reduction in~\cite[Section~2]{logunov2020landis},  the equation~\eqref{2d-elliptic} in nondivergence form can be rewritten in divergence form~\eqref{pos}. Precisely speaking, by the solution constructed in Lemma~\ref{lma-ode} and an elementary computation, for solutions of~\eqref{2d-elliptic} defined on $\displaystyle D_3=[-2,3]\times \left[-5/2,5/2\right]$, we deduce that
\begin{equation}\label{reduction}
    \begin{cases}
    {\color{review1}\partial_y \Phi|_{y=0}=0}&\Longrightarrow \partial_y\left.\left(\frac{\Phi(z)}{\varphi(x)}\right)\right|_{y=0}=0,\\
    -\Delta \Phi(z)+V(x)\Phi(z)=0&\Longrightarrow \nabla\cdot \left(\varphi^2\nabla \bigl(\frac{\Phi(z)}{\varphi(x)}\bigr)\right)=0.
    \end{cases}
\end{equation}
Combining the above reduction with Theorem~\ref{pos-rmk-simplified}, the following $L^2$-propagation of smallness is established in \cite[Proposition~3.6]{su2023quantitative}:
\begin{proposition}\label{propagation-prp}
	Let $C_0>0$ be a positive constant. Then for any measurable set $\omega\subset I=[0,1]$ with $|\omega|\in (0,\frac{1}{2})$, any potential $V \in L^\infty(\R)$ with $0<V(x)\le C_0$ for any $x\in I_4=[-4\sqrt{2},4\sqrt{2}]$, and any real-valued $H^2_{\mathrm{loc}}$ solution $\Phi$ of~\eqref{2d-elliptic} in $\R^2$, we have
	\begin{equation}
		\|\Phi\|_{L^2(D_1)}\le C\|\Phi\|^{\alpha}_{L^2(\omega)}\left( \sup_{D_2}|\Phi|^{1-\alpha} \right),\label{uniform-propagation} 
	\end{equation}
	where $D_1=[0,1]\times[-1/2,1/2]$,  $D_2=[-1,2]\times [-3/2,3/2]$. Furthermore, there exists a numerical positive constant $d>0$ such that
	{\color{myself}\begin{equation}
    C\le \exp\bigl(d\exp(d C_0^{\frac{1}{2}})\bigr)\label{estimate-new-1}
\end{equation}}
and 
\begin{equation}
   \alpha\ge  \frac{\exp\bigl(-d\exp(dC_0^{\frac{1}{2}})\bigr)}{\bigl|\log|\omega|\bigr|^2},\label{estimate-new-2}.
\end{equation}
\end{proposition}
The original proposition in \cite{su2023quantitative} only mentioned that the constant depends on $V$ and $|\omega|$. However, if we check the proof step-by-step, the constant only depends on the $L^\infty$-norm of $V$ and $|\omega|$.
To see how the constants $C$ and $\alpha$ depend on $\|V\|_{L^\infty}$ and $|\omega|$ exactly, we present the proof here.
\begin{proof}[Proof of Proposition~\ref{propagation-prp}]
We first establish the $L^\infty$-propagation of smallness, \textit{i.e.}, for any measurable set $\omega \subset I$ with $|\omega|\in(0,\frac{1}{2})$ and any $H^2_{\mathrm{loc}}$ solution $\phi$ of~\eqref{2d-elliptic},
\begin{equation}\label{d-2}
    \sup_{D_1}|\Phi|\le C\left(\sup_{\omega}|\Phi|^\alpha\right)\left(\sup_{D_2}|\Phi|^{1-\alpha}\right),
\end{equation}
where $\alpha$ and $C$ satisfy
\eqref{estimate-alpha} and~\eqref{estimate-c}
where $\Lambda=e^{d_3 C_0^{\frac{1}{2}}}$, and $d_1,d_2,d_3$ are numerical positive constants.
The numerical constants may be different from line to line.
Indeed, we consider the positive $W^{1,\infty}$ solution $\varphi$ constructed in Lemma~\ref{lma-ode} for the second-order ODE~\eqref{ode} on $I_4$. Then from the reduction~\eqref{reduction}, we can apply Theorem~\ref{pos-rmk-simplified} to $(\Phi/\varphi)$ with $\ell_0=\left\{(x,y)\in\R^2:y=0\right\}$. According to the scaling method, we first replace $B_1$ and $B_2$ in Theorem~\ref{pos-rmk-simplified} with
$B_1':=B_{\frac{\sqrt{2}}{2}}\left(\frac{1}{2}\right)$  and  $B_2':=B_{\sqrt{2}}\left(\frac{1}{2}\right)$, then we have
\begin{equation}\label{d-1}
    \sup_{B_1'}\left|\frac{\Phi}{\varphi}\right|\le C\left(\sup_{\omega}\left|\frac{\Phi}{\varphi}\right|^\alpha\right)\left(\sup_{B_2'}\left|\frac{\Phi}{\varphi}\right|^{1-\alpha}\right),
\end{equation}
where $\alpha$ and $C$ satisfy~\eqref{estimate-alpha} and~\eqref{estimate-c} with $\Lambda=e^{d_3 C_0^{\frac{1}{2}}}$ by Lemma~\ref{lma-ode}, $d_1,d_2,d_3>0$ are numerical positive constants and $d_1,d_2$ may be possibly different from those in~\eqref{estimate-alpha} and~\eqref{estimate-c}.

On the other hand, we have the following inclusion relation
\begin{equation}
    D_1\subset B_1'\subset B_2'\subset D_2.
\end{equation}
Therefore, we obtain from~\eqref{d-1} and Lemma~\ref{lma-ode} that
\begin{equation}
    \sup_{D_1}|\Phi|\le Ce^{C_0^{\frac{1}{2}}}\left(\sup_\omega|\Phi|^\alpha \right)\left(\sup_{D_2}|\Phi|^{1-\alpha}\right),
\end{equation}
where $\alpha$ and $C$ satisfy~\eqref{estimate-alpha} and~\eqref{estimate-c} with $\Lambda =e^{d_3 C_0^{\frac{1}{2}}}$ and $d_1,d_2,d_3>0$.
By choosing new $d_1,d_2$ appropriately, we complete the proof of~\eqref{d-2} with the constants given by the forms which are the same as those in~\eqref{estimate-alpha} and~\eqref{estimate-c}.

Then we replace $L^\infty$-norms with $L^2$-norms. From~\eqref{d-2}, there exist some constants $\alpha_1\in (0,1)$, such that for any $\widetilde{\omega}\subset \omega$ with $|\omega|/2\le |\widetilde{\omega}|\le |\omega|$, we have
\begin{equation}\label{d-3}
    \sup_{D_1}|\Phi|\le C_1\left(\sup_{\widetilde{\omega}}|\Phi|^{\alpha_1}\right)\left(\sup_{D_2}|\Phi|^{1-\alpha_1}\right),
\end{equation}
where $\alpha_1$ and $C_1$ are given by~\eqref{estimate-alpha} and~\eqref{estimate-c} with $\Lambda=e^{d_3C_0}$ for some numerical positive constants $d_1,d_2$ and $d_3$.

Let $\varepsilon>0$ be a small constant to be chosen later. Assume $\phi\not\equiv 0$ and define $a$ by
\begin{equation}
    a:=\left(\varepsilon \frac{\sup_{D_1}|\Phi|}{\sup_{D_2}|\Phi|^{1-\alpha_1}} \right)^{\frac{1}{\alpha_1}}
\end{equation}
and the set
\begin{equation}
    \omega_a:=\left\{x\in \omega: |\Phi(x,0)|\le a\right\}.
\end{equation}
When $\varepsilon C_1<1$, we must have $|\omega_a|\le |\omega|/2 $. Indeed, otherwise~\eqref{d-3} would hold with {\color{review1}$\widetilde{\omega}$} replaced by $\omega_a$, leading to
\begin{equation}
    \sup_{D_1}|\Phi|\le C_1\left(\sup_{\omega_a}|\Phi|^{\alpha_1}\right)\left(\sup_{D_2}|\Phi|^{1-\alpha_1}\right)=C_1\varepsilon \left(\frac{\sup_{\omega_a}|\Phi|}{a}\right)^{\alpha_1}\sup_{D_1}|\Phi|<\sup_{D_1}|\Phi|.
\end{equation}
This is a contradiction with $\Phi\not\equiv 0$. We choose $\varepsilon=\frac{1}{2C_1}$ for simplicity, then $\varepsilon=\varepsilon(C_0,\alpha)$. As a consequence, 
\begin{equation}
    \|\Phi\|^2_{L^2(\omega)}\ge \|\Phi\|^2_{L^2(\omega\backslash\omega_a)}\ge \frac{a^2}{2}|\omega|= \frac{1}{2}|\omega|\left(\frac{1}{2C_1}\cdot \frac{\sup_{D_1}|\Phi|}{\sup_{D_2}|\Phi|^{1-\alpha_1}} \right)^{\frac{2}{\alpha_1}},
\end{equation}
and we deduce
\begin{equation}
    \sup_{D_1}|\Phi|\le 2C_1\left(\frac{2}{|\omega|}\right)^{\frac{\alpha_1}{2}}\|\Phi\|^{\alpha_1}_{L^2(\omega)}\left(\sup_{D_2}|\Phi|^{1-\alpha_1}\right).
\end{equation}
Recall that $\alpha_1$ and $C_1$ are estimated in~\eqref{estimate-alpha} and~\eqref{estimate-c}, we have
\begin{equation}
    \left(\frac{2}{|\omega|}\right)^{\frac{\alpha_1}{2}}\le \exp\biggl(\frac{d_4\exp(-d_4^{-1}\Lambda^2)}{\bigl|\log|\omega|\bigr|}\biggr)\le \exp\biggl(\frac{d_4\exp(-d_4^{-1}\Lambda^2)}{\log 2}\biggr)\le d_5
\end{equation}
for some sufficiently large positive numerical constants $d_4$ and $d_5$. This estimate completes the proof by choosing $d$ in the proposition large enough.
\end{proof}

Last, we give the $L^2$--propagation of smallness for $H^2_{\mathrm{loc}}$ solution of~\eqref{2d-elliptic} in each domain $D_{3,n}$ under Assumption~\ref{assump1}:
\begin{corollary}\label{propagation-crc}
	Let $\gamma_n\in(0,\frac{1}{2}),n\in\Z$ and $V\in L^\infty_{\mathrm{loc}}(\R)$ be a potential as in Assumption~\ref{assump1}. Let $\Omega$ be a measurable subset of \,$\R$, $\omega_n:=\Omega\cap I_n$ such that $|\omega_n|= \gamma_n|I_{n}|$ for every $n\in \Z$. Then for any real-valued $H^2_{\mathrm{loc}}$ solution $\Phi$ of~\eqref{2d-elliptic} in $\R^2$, we have
	\begin{equation}
		\|\Phi\|_{L^2(D_{1,n})}\le C a_n^{\frac{\alpha_n-2}{2}}\|\Phi\|^{\alpha_n}_{L^2(\omega_n)}\left( \sup_{D_{2,n}}|\Phi|^{1-\alpha_n} \right),\quad \forall n\in \Z, 
	\end{equation}
where $C=C(c_2,\beta_2)$ depends only on $c_2$ and $\beta_2$, $a_n=|I_n|^{-1}$, and $\alpha_n$ satisfies
	\begin{equation}
	    \alpha_n\ge \frac{1}{d} \frac{1}{|\log\gamma_n|^2}
	\end{equation}
	with $d:=d(c_2,\beta_2)>0$ depending only on $c_2$ and $\beta_2$.
\end{corollary}

\begin{proof}
	For each $n\in \Z$, we first reduce~\eqref{2d-elliptic} to uniformly bounded potential: define 
	\begin{equation}\label{transformation}
		f(z):=\Phi\left( \frac{z}{a_n} \right), \quad \forall z\in a_nD_{3,n},\, \forall n\in \Z
	\end{equation}
	and substitute this into~\eqref{2d-elliptic}, then we obtain
	\begin{equation}
		-\Delta f(z)+\widetilde{V}(x)f(z)=0,\quad \forall z\in a_n D_{3,n},\,\forall n\in \Z\label{reduced-2d-elliptic}
	\end{equation}
	where $\displaystyle \widetilde{V}(x):=\frac{1}{a_n^2}V\left( \frac{x}{a_n} \right)$.
	By Assumption~\ref{assump1}, the new potential satisfies the condition
	\begin{equation}
	\widetilde{V}(x)\le c_2 \frac{1}{a_n^2}\left\langle \frac{x}{a_n}\right\rangle^{\beta_2}
	\end{equation}
	for all $x\in a_nI_{3,n}$ and for all $n\in \Z$.

By definition, $a_n=|I_n|^{-1}= |x_n|^s$
for all $n\in \Z$.
By Lemma~\ref{growing-estimate-lemma}, for any \textcolor{review1}{$x\in I_{3,n}$} and $n$ sufficiently large, we have
\begin{equation}
    \frac{x}{a_n} \lesssim  x_n+3\frac{1}{x_n^s}\lesssim 2 x_n.
\end{equation}
	Then we obtain that for any \textcolor{review1}{$x\in I_{3,n}$}, there exists a constant $C'$ such that
	\begin{equation}\label{g-2}
		\widetilde{V}(x)\le c_2 \frac{1}{a_n^2}\left\langle \frac{x}{a_n}\right\rangle^{\beta_2}\le C' |x_n|^{\beta_2-2s}
	\end{equation}
	hold uniformly for all $n\in \Z$ and $C'=C'(c_2,\beta_2)$ depends only on $c_2$ and $\beta_2$. Since $s\ge\beta_2/2$, the estimate~\eqref{g-2} can be simplified to
	\begin{equation}\label{g-simple-version}
	    		\widetilde{V}(x)\le C',
	\end{equation}
	where $C':=C'(c_2,\beta_2)$ depends only on $c_2$ and $\beta_2$.
	That is to say, we have a uniform bound for all $n\in\Z$.

Now we can use Proposition~\ref{propagation-prp} for~\eqref{reduced-2d-elliptic}, then for any $n\in \Z$ and real-valued $H^2_{\mathrm{loc}}$ solution $f$ of~\eqref{reduced-2d-elliptic} in $a_nD_{3,n}$, we have
	\begin{equation}
		\|f\|_{L^2(a_nD_{1,n})}\le C\|f\|^{\alpha_n}_{L^2(a_n\omega_n)}\left( \sup_{a_nD_{2,n}}|f|^{1-\alpha_n} \right),\label{aux-propagation} 
	\end{equation}
	where $C:=C(c_2,\beta_2)$ depends only on $c_2$ and $\beta_2$, and $\alpha_n$ satisfies
	\begin{equation}
	    \alpha_n\ge \frac{1}{d}\frac{1}{|\log\gamma_n|^2} 
	\end{equation}
	with $d:=d(c_2,\beta_2)>0$ depending only on $c_2$ and $\beta_2$.
	
	Notice that 
	\begin{equation}\label{g-4}
		\|f\|_{L^2(a_nD_{1,n})}=a_n\|\Phi\|_{L^2(D_{1,n})}
	\end{equation}
	and
	\begin{equation}\label{g-5}
		\|f\|^{\alpha_n}_{L^2(a_n \omega_n)}=a_n^{\frac{\alpha_n}{2}}\|\Phi\|^{\alpha_n}_{L^2(\omega_n)}
	\end{equation}
	where the first one is a 2D integral and the second one is a 1D integral.
	For the $L^\infty$-norm, we simply have
	\begin{equation}\label{g-6}
		\sup_{a_nD_{2,n}}|f|^{1-\alpha_n}=\sup_{D_{2,n}}|\Phi|^{1-\alpha_n}.
	\end{equation}
	Take the above three equations into~\eqref{aux-propagation}, then we obtain 
	\begin{equation}\label{a-1}
		\|\Phi\|_{L^2(D_{1,n})}\le C a_n^{\frac{\alpha_n-2}{2}}\|\Phi\|^{\alpha_n}_{L^2(\omega_n)}\left( \sup_{D_{2,n}}|\Phi|^{1-\alpha_n} \right). 
	\end{equation}
	Hence we have finished the proof.
\end{proof}

\subsection{Proof of Theorem~\ref{spectral-inequality}}\label{subsec-3-4}
In this subsection, we present the proof of Theorem~\ref{spectral-inequality}. 

We make use of the \textit{ghost dimension} construction, which was first introduced in \cite{jerison1999nodal}. Without loss of generality, we assume that $\lambda>1$. For any $\lambda>1$, $z=(x,y) \in \R\times \left[ -{5}/{2},{5}/{2} \right] $, with $\phi \in \mathcal{E}_\lambda(\sqrt{H})$ as
	\begin{equation}
		\phi=\sum_{\lambda_k\le \lambda} b_k \phi_k,\quad b_k \in \C,
	\end{equation}
	we set
	\begin{equation}
		\Phi(x,y):=\sum_{\lambda_k\le\lambda}b_k\cosh(\lambda_k y)\phi_k(x).\label{a-2}
	\end{equation}
		For the case of  $y=0$, using the fact that $(\cosh s)'=\sinh s$, we obtain
	\begin{equation}
		\partial_y \Phi\lvert_{y=0}=\sum_{\lambda_k\le \lambda}\lambda_kb_k\sinh(\lambda_k 0)\phi_k=0.
	\end{equation}
	On the other hand, taking the derivative of~\eqref{a-2} with respect to $y$ twice and using the fact that $(\cosh s)''=\cosh s$, we obtain
	\begin{equation}
		\partial_{y}^2 \Phi=H\Phi=\sum_{\lambda_k\le \lambda}\lambda_k^2b_k\cosh(\lambda_ky)\phi_k.
	\end{equation}
	It follows that
	\begin{equation}
		-\Delta \Phi+V(x)\Phi=-\partial^2_{y}\Phi+H\Phi=0.
	\end{equation}
	Hence the function $\Phi$ is a $H^2_{\mathrm{loc}}$ solution for~\eqref{2d-elliptic} with $\Phi(x,0)=\phi(x)$ on $\R$.

	Applying Corollary~\ref{propagation-crc} to $\Phi$, we obtain
	\begin{equation}
		\|\Phi\|_{L^2(D_{1,n})}\le C_1a_n^{\frac{\alpha_n-2}{2}}\|\phi\|^{\alpha_n}_{L^2(\omega_n)}\left( \sup_{D_{2,n}}|\Phi|^{1-\alpha_n} \right),\quad \forall n\in \Z, 
	\end{equation}
	where $C_1=C_1(c_2,\beta_2)\lesssim_{c_2,\beta_2} 1$ and $\alpha_n=\alpha_n(c_2,\beta_2)\gtrsim_{c_2,\beta_2} \frac{1}{\left|\log \gamma_n\right|^2}$.
	From Young's inequality for products, {\it{i.e.}}, $\displaystyle ab\le \alpha_n a^{\frac{1}{\alpha_n}}+(1-\alpha_n) b^{\frac{1}{1-\alpha_n}}$ for any $a,b\ge 0$, we have for all $n\in \Z$ and all $\varepsilon>0$
	\begin{equation}\label{a-3}
		\|\Phi\|^2_{L^2(D_{1,n})}\le \frac{C_1 a_n^{\frac{\alpha_n-2}{2}} \alpha_n}{\varepsilon }\|\phi\|^2_{L^2(\omega_n)}+C_1a_n^{\frac{\alpha_n-2}{2}}(1-\alpha_n)\varepsilon ^{\frac{\alpha_n}{1-\alpha_n}}\|\Phi\|^2_{L^{\infty}(D_{2,n})}.
	\end{equation}

Thanks to Lemma~\ref{localization}, we only need to consider the value of $\phi$ in $I_\lambda$. Define
	\begin{equation}
		\mathcal{J}:=\left\{n\in \Z: I_{n}\cap I_\lambda\neq \varnothing\right\} \text{ and } n_0:=\max_{n\in\mathcal{J}}n.
	\end{equation}
Notice that for all $n\in\mathcal{J}$ we have
	\begin{equation}
		a_n=|x_n|^{s}\le C_2\langle \lambda\rangle ^{\frac{2s}{\beta_1}},\label{a-5}
	\end{equation}
	and
	\begin{equation}\label{a-9}
	    \alpha_n\gtrsim \frac{1}{|\log \gamma_n|^2}\asymp \frac{1}{(1+|x_n|^{\tau})^2\left|\log\gamma\right|^2}\Longrightarrow \alpha_n\ge C_2 \langle \lambda \rangle^{-\frac{4\tau }{\beta_1}}\frac{1}{|\log\gamma|^2},
	\end{equation}
	where $C_2=C_2(c_1,\beta_1)$ depends only on $c_1$ and $\beta_1$, while $\tau\ge 0$ is from the assumption of $\Omega$ given in ~\eqref{eqn-by-Ming}. 
 Since $n_0\in\mathcal{J}$, we have
 \begin{equation}
     \alpha_{n_0}\ge C_2 \langle \lambda \rangle^{-\frac{4\tau }{\beta_1}}\frac{1}{|\log\gamma|^2}.\label{aa-9}
 \end{equation}
Summing over $\mathcal{J}$ for~\eqref{a-3}, we obtain
\begin{equation}\label{a-8}
    	\sum_{n\in \mathcal{J}}\|\Phi\|^2_{L^{2}(D_{1,n})}\le \frac{C_1}{\varepsilon }\sum_{n\in \mathcal{J}}\|\phi\|^2_{L^2(\omega_n)}+C_1\varepsilon ^{\frac{\alpha_{n_0}}{1-\alpha_{n_0}}}\sum_{n\in \mathcal{J}} \|\Phi\|^2_{L^\infty(D_{2,n})},
\end{equation}
where we have used the inequality $a_n^{\frac{\alpha_n-2}{2}}\le 1$. 
		
For the second term on the right-hand side of~\eqref{a-8}, we have the following lemma:
\begin{lemma}\label{lma-different}
There exists a positive constant $C_3:=C_3(V)>0$ depending only on the potential $V$ such that 
\begin{equation}\label{a-7}
    \sum_{n\in\mathcal{J}}\|\Phi\|^2_{L^\infty(D_{2,n})}\le C_3e^{6\lambda}\|\phi\|^2_{\R}.
\end{equation}
\end{lemma}

We postpone the proof of Lemma~\ref{lma-different} to the end of this subsection.

\begin{proof}[Proof of Theorem~\ref{spectral-inequality}]
We need to estimate the lower bound of the left-hand side of~\eqref{a-8}, 
\begin{equation}\label{a-6-6}
	\begin{aligned}
	\sum_{n\in \mathcal{J}}\|\Phi\|^2_{L^2(D_{1,n})} &\ge \int_{-|I_{n_0}| /2}^{|I_{n_0}| /2} \int_{I_{\lambda}}|\Phi(x,y)|^2 \d x\mathrm{d}y\\
							&\ge \frac{1}{2}\int_{-|I_{n_0}| /2}^{|I_{n_0}| /2}\int_{\R}|\Phi(x,y)|^2\d x \mathrm{d}y\\
					&=\int_0^{|I_{n_0}| /2}\int_{\R}\biggl|\sum_{\lambda_k\le \lambda}b_k\cosh(\lambda_k y)\phi_k(x)\biggr|^2\d x \mathrm{d}y,
\end{aligned}
\end{equation}
where the second inequality uses Lemma~\ref{localization}. By the orthogonality of {\color{review1}$\phi_k$}, we obtain
\begin{equation}\label{a-6}
	\begin{aligned}
	\sum_{n\in \mathcal{J}}\|\Phi\|^2_{L^2(D_{1,n})}
				&=\int_0^{|I_{n_0}| /2} \sum_{\lambda_k\le \lambda}|b_k|^2\cosh(\lambda_ky)^2\d x \mathrm{d}y\\
							&\ge \frac{|I_{n_0}|}{2}\sum_{\lambda_k\le \lambda}|b_k|^2\ge \frac{1}{2C_2}\langle \lambda\rangle ^{-\frac{2s}{\beta_1}} \|\phi\|^2_{L^2(\R)}
,
\end{aligned}
\end{equation}
where the last inequality uses~\eqref{a-5}. 
Taking~\eqref{a-7} and~\eqref{a-6} into~\eqref{a-8}, we obtain
\begin{equation}\label{a-10}
	\|\phi\|^2_{L^2(\R)} \le \frac{C_4}{\varepsilon }\langle \lambda\rangle ^{\frac{2s}{\beta_1}}\|\phi\|^2_{L^2(\Omega)}+C_4 \varepsilon ^{\frac{\alpha_{n_0}}{1-\alpha_{n_0}}}\langle \lambda\rangle ^{\frac{2s}{\beta_1}}e^{6 \lambda}\|\phi\|^2_{L^2(\R)},
\end{equation}
where $C_4:=C_4(V)>0$ depends only on $V$.

Now choose
\begin{equation}
    \varepsilon=\left(\frac{1}{2C_4 \langle\lambda\rangle^{\frac{2s}{\beta_1}}e^{6\lambda}}\right)^{\frac{1-\alpha_{n_0}}{\alpha_{n_0}}} \Longleftrightarrow C_4 \varepsilon ^{\frac{\alpha_{n_0}}{1-\alpha_{n_0}}}\langle \lambda\rangle ^{\frac{2s}{\beta_1}}e^{6\lambda}=\frac{1}{2}
\end{equation}
in~\eqref{a-10}. From~\eqref{aa-9} we have
\begin{equation}
    \frac{1}{\varepsilon}\le \left(2C_4 \langle\lambda\rangle^{\frac{2s}{\beta_1}}e^{6\lambda}\right)^{\frac{1}{\alpha_{n_0}}}\le C_5 e^{C_5 \lambda^{\frac{4\tau }{\beta_1}+1}}
\end{equation}
where $C_5:=C_5(V)>0$ depends only on $V$.
Substituting this $\varepsilon$ into~\eqref{a-10}, we obtain
\begin{equation}
    \|\phi \|^2_{L^2(\R)}\le  \frac{2C_4}{\varepsilon }\langle \lambda\rangle ^{\frac{2s}{\beta_1}}\|\phi\|^2_{L^2(\Omega)}\le C'\left( \frac{1}{\gamma}\right)^{C'\lambda^{\frac{4\tau }{\beta_1}+1}} \|\phi\|^2_{L^2(\Omega)}
\end{equation}
where $C'=C'(V)$ depends only on $V$. This finishes the proof of Theorem~\ref{spectral-inequality}.
\end{proof}

Finally, we give the proof of Lemma~\ref{lma-different}:
\begin{proof}[Proof of Lemma~\ref{lma-different}]
	We take a $\mathcal{C}^2$ cut-off function $\chi:\R^2\to \R$ such that
	\begin{equation}
		\chi\equiv 1 \,\text{ on } \left[ -\frac{3}{2},\frac{3}{2} \right] ^2 \, \text{ and } \,\,\mathrm{supp}\chi \subset \left[ -\frac{5}{2},\frac{5}{2} \right] ^2.
	\end{equation}
	For any $n \in \Z$, we set
	\begin{equation}
		\chi_{n}(x,y):=\chi\left(a_n(x-x_n)+\frac{1}{2},a_n y\right).
	\end{equation}
By this setting, for any $n\in\Z$ we have
\begin{equation}
    \chi_n\equiv 1 \text{ on } D_{2,n} \text{ and } \mathrm{supp}\chi_{n}\subset D_{3,n}.
\end{equation}
	There exists a positive constant $C>0$ which depends only on the choice of $\chi$, such that
	 \begin{equation}
		|D\chi|_{L^{\infty}(\R^2)}\le C \,\text{ and }\,|\mathrm{Hess}\chi|_{L^{\infty}(\R^2)}\le C.
	\end{equation}
	Then after the rescaling, we have
	\begin{equation}\label{rescaling-estimate}
	    |D\chi_{n}|_{L^\infty(\R^2)}\le C a_n \text{ and } |\mathrm{Hess}\chi_{n}|_{L^\infty(\R^2)}\le C a_n^2.
	\end{equation}
	where $C>0$ depends only $c_2,\beta_2$ and the choice of $\chi$. 

Using the 2D Sobolev embedding theorem, we obtain
	\begin{align}\label{2dsobolev}
		\sum_{n\in\mathcal{J}}\|\Phi\|^2_{L^\infty(D_{2,n})}&\le \sum_{n\in\mathcal{J}}\|\chi_n\Phi\|^2_{L^\infty(\R^2)}\le \pi \sum_{n\in\mathcal{J}}\|\chi_n\Phi\|^2_{H^2(\R^2)}.
	\end{align}
	By the definition of the $H^2$-norm, we have 
	\begin{equation}
	\begin{aligned}
	    \sum_{n\in\mathcal{J}}\|\chi_n\Phi\|^2_{H^2(\R^2)}
	    =&\underbrace{\sum_{n\in\mathcal{J}}\|\chi_n\Phi\|^2_{L^2(\R^2)}}_{\rm S_1}+\underbrace{\sum_{n\in\mathcal{J}}\|D_x(\chi_n \Phi)\|^2_{\textcolor{myself}{L^2(\R^2)}}}_{\rm S_2}+\underbrace{\sum_{n\in\mathcal{J}}\|D_y(\chi_n \Phi)\|^2_{\textcolor{myself}{L^2(\R^2)}}}_{\rm S_3}\\
	    +&\underbrace{\sum_{n\in\mathcal{J}}\|D_x^2({\color{review1}\chi_n} \Phi)\|^2_{L^2(\R^2)}}_{\rm S_4}+\underbrace{\sum_{n\in\mathcal{J}}\|D_y^2(\chi_n \Phi)\|^2_{L^2(\R^2)}}_{\rm S_5}\\
	    =& {\rm S_1+S_2+S_3+S_4+S_5}.
	    \end{aligned}
	\end{equation}
	We need to check that each one of the five terms is bounded by $e^{5\lambda}$ multiplied by a positive constant.
	
	\medskip
	\noindent{\bf Estimate of $\rm S_1$.} This term is the simplest one, just extending the integral from $\displaystyle \sum_{n\in\mathcal{J}}\int_{I_{3,n}}$ to $\displaystyle\int_{\R}$ and using the orthogonality of eigenfunctions, we obtain
	\begin{equation}\label{estimate-1}
	    \begin{aligned}
	    J_1&\le \sum_{n\in\mathcal{J}} \int_{-\frac{5}{2}}^{\frac{5}{2}} \int_{I_{3,n}}\bigl|\chi_n(x,y)\sum_{\lambda_k\le\lambda}b_k\cosh (\lambda_k y) \phi_k(x) \bigr|^2\d x\mathrm{d}y\\
	    &\le \kappa_1 \int_{-\frac{5}{2}}^{\frac{5}{2}}\int_{\R}\bigl| \sum_{\lambda_k\le\lambda}b_k\cosh (\lambda_k y) \phi_k(x) \bigr|^2\d x\mathrm{d}y\\
	    &\le 5\kappa_1 e^{5\lambda}\|\phi\|^2_{L^2(\R)},
	    \end{aligned}
	\end{equation}
	where the numerical positive constant $\kappa_1>0$ comes from the overlapping between the $I_{3,n}$'s.

		\medskip
  
	\noindent{\bf Estimate of $\rm S_2$.} The estimate of this term heavily relies on the potential $V$. Using the inequality $(a+b)^2\le 2(a^2+b^2)$, we obtain
	    \begin{align}
	    \mathrm{S}_2&\le 2\sum_{n\in\mathcal{J}}\int_{-\frac{5}{2}}^{\frac{5}{2}}\int_{I_{3,n}}\bigl|D_x\chi_n(x,y)\sum_{\lambda_k\le \lambda} b_k\cosh(\lambda_k y)\phi_k(x)\bigr|^2\d x\mathrm{d}y\\
	    &+2\underbrace{\sum_{n\in\mathcal{J}}\int_{-\frac{5}{2}}^{\frac{5}{2}}\int_{I_{3,n}}\bigl| \chi_n(x,y) \sum_{\lambda_k\le \lambda} b_k\cosh(\lambda_ky)D_x\phi_k(x) \bigr|^2\d x\mathrm{d}y}_{J}\label{eqn-I-defined}\\
	    &\le 10\kappa_1 C a_n^2 e^{5\lambda}\|\phi\|^2_{L^2(\R)}+2J,
	    \end{align}
	where the estimate of the first term comes from~\eqref{rescaling-estimate} and is similar to~\eqref{estimate-1}. The second term $2J$ is more subtle, we denote 
	\begin{equation}
	    N(\lambda):= \#\left\lbrace\lambda_k :\lambda_k\le \lambda \right\rbrace.
	\end{equation}
	Notice that, from 
	\begin{equation}
	    N(\lambda)\le \sum_{k=1}^{N(\lambda)}\left((\lambda+1)^2-\lambda_k^2\right)
	\end{equation}
	and the lower bound $V(x)\ge c_1|x|^{\beta_1}$, the right-hand side can be estimated explicitly using the classic Lieb-Thirring bound from~\cite[Theorem~1]{lieb2001inequalities}. More precisely, for $\lambda>0$ we have
	\begin{equation}\label{number-eigens}
	    \begin{aligned}
	    \sum_{k=1}^{N(\lambda)}\left((\lambda+1)^2-\lambda_k^2\right)&\lesssim \int_{\R}\max\lbrace (\lambda+1)^2-V(x),0 \rbrace^{\frac{1}{2}+1}\d x\\
	    &\le \int_{I_{\left((\lambda+1)^2/c_1\right)^{1 /\beta_1}}}(\lambda+1)^{3}\d x\\
	    &\lesssim_{c_1}(\lambda+1)^{\frac{2+3\beta_1}{\beta_1}}.
	    \end{aligned}
	\end{equation}
	Then we can estimate $J$ in~\eqref{eqn-I-defined} as
	\begin{equation}
	    \begin{aligned}
	       J&\le \sum_{n\in\mathcal{J}}\int_{-\frac{5}{2}}^{\frac{5}{2}}N(\lambda)\sum_{\lambda_k\le \lambda}|b_k\cosh(\lambda_k y)|^2 \int_{I_{3,n}}|D_x\phi_k(x)|^2\d x\mathrm{d}y\\
	       &\le \sum_{n\in\mathcal{J}}\int_{-\frac{5}{2}}^{\frac{5}{2}} N(\lambda) \left(1+\frac{8}{|I_{3,n}|^2}\right)\sum_{\lambda_k\le\lambda}(1+\lambda_k^2)|b_k\cosh(\lambda_k y)|^2\int_{2I_{3,n}}|\phi_k(x)|^2\d x\mathrm{d}y.
	    \end{aligned}
	\end{equation}
 where the second inequality has used~\eqref{caccio-1} given by Lemma~\ref{grd-estimate}.
	Notice that $N(\lambda)\lesssim_{c_1} \langle \lambda\rangle^{\frac{2+3\beta_1}{\beta_1}}$ from~\eqref{number-eigens} , $|I_{3,n}|^{-1}\lesssim \langle a_n\rangle\lesssim \langle \lambda\rangle ^{\frac{2s}{\beta_1}}$ from~\eqref{a-5},  
	$1+\lambda_k^2\le 1+\lambda^2$ and $\cosh(\lambda_ky)\lesssim e^{5\lambda}$, we have
	\begin{equation}\label{eqn-I-estimated}
	\begin{aligned}
	    J&\le \kappa_2 e^{5\lambda} \sum_{\lambda_k\le\lambda}|b_k|^2\sum_{n\in\mathcal{J}} \int_{2I_{3,n}}|\phi_k(x)|^2\d x\mathrm{d}y\\
	    &\le \kappa_3 e^{5\lambda}\|\phi\|_{L^2(\R)},
	    \end{aligned}
	\end{equation}
	where $\kappa_2$ and $\kappa_3$ are numerical positive constants.
	Hence we obtain from~\eqref{eqn-I-defined} and~\eqref{eqn-I-estimated} that
	\begin{equation}
	    \mathrm{S}_2\le \kappa_4 e^{5\lambda}\|\phi\|^2_{L^2(\R)}
	\end{equation}
	holds for some numerical positive constant $\kappa_4>0$.

		\medskip
  
	\noindent{\bf Estimate of $\rm S_3$.}
	This term is similar to the first one, just expanding the integral $\displaystyle\sum_{n\in\mathcal{J}}\int_{I_{3,n}}$ to $\displaystyle\int_{\R}$ and using~\eqref{rescaling-estimate},  we have
	\begin{equation}
	    \begin{aligned}
	    \mathrm{S}_3&\le\kappa_1\int_{-\frac{5}{2}}^{\frac{5}{2}}\int_{\R}\bigl|D_y\chi_n(x,y)\sum_{\lambda_k\le \lambda} b_k\cosh(\lambda_k y)\phi_k(x)\bigr|^2\d x\mathrm{d}y\\
	    &+ \kappa_1\int_{-\frac{5}{2}}^{\frac{5}{2}} \int_{\R}\bigl|\chi_n(x,y)\lambda_k\sum_{\lambda_k\le\lambda}b_k\sinh (\lambda_k y) \phi_k(x) \bigr|^2\d x\mathrm{d}y\\
	    &\le 5\kappa_1 (a_n^2+\lambda^2) e^{5\lambda}\|\phi\|^2_{L^2(\R)},
	    \end{aligned}
	\end{equation}
	where the last step uses the orthogonality of eigenfunctions.

 \medskip
	
	\noindent{\bf Estimate of $\rm S_4$.}
	Notice that 
	\begin{equation}
	\begin{aligned}
		    \mathrm{S}_4&=\sum_{n\in\mathcal{J}}\|(D_x^2\chi_n)\Phi+2(D_x\chi_n)D_x\Phi+\chi_n D_x^2\Phi\|^2_{L^2(\R^2)}\\
		    &\le 3\sum_{n\in\mathcal{J}}\left(\|(D_x^2\chi_n)\Phi\|^2_{L^2(\R^2)}+4\|(D_x\chi_n)D_x\Phi\|^2_{L^2(\R^2)}+\|\chi_n D_x^2\Phi\|^2_{L^2(\R^2)}\right).
	\end{aligned}
	\end{equation}
	The upper bounds of the first and second terms in the above inequality can be obtained easily by the same methods used in the estimates of $\mathrm{S}_1$ and $\mathrm{S}_2$ respectively, that is, there exists a positive constant $\kappa_5>0$ such that
	\begin{equation}
	    \sum_{n\in\mathcal{J}}\left(\|(D_x^2\chi_n)\Phi\|^2_{L^2(\R^2)}+4\|(D_x\chi_n)D_x\Phi\|^2_{L^2(\R^2)}\right)\le \kappa_5 (a_n^4+a_n^2) e^{5\lambda}\|\phi\|_{L^2(\R)}.
	\end{equation}

	Hence we only need to estimate the third term, which can be bounded as follows
	\begin{equation}
	\begin{aligned}
	    \sum_{n\in\mathcal{J}}\|\chi_nD^2_x\Phi\|^2_{L^2(\R^2)}&\le\sum_{n\in\mathcal{J}}\int_{-\frac{5}{2}}^{\frac{5}{2}}\int_{I_{3,n}}\bigl|\sum_{\lambda_k\le \lambda} b_k\cosh(\lambda_k y)D_x^2\phi_k\bigr|^2\d x\mathrm{d}y\\
	    &=\sum_{n\in\mathcal{J}}\int_{-\frac{5}{2}}^{\frac{5}{2}}\int_{I_{3,n}}\bigl|\sum_{\lambda_k\le \lambda} b_k\cosh(\lambda_k y)\left(V(x)-\lambda_k^2\right)\phi_k\bigr|^2\d x\mathrm{d}y\\
	    &\le \sum_{n\in\mathcal{J}}\int_{-\frac{5}{2}}^{\frac{5}{2}}\int_{I_{3,n}}N(\lambda)\sum_{\lambda_k\le \lambda}\bigl| b_k\cosh(\lambda_k y)\left(V(x)-\lambda_k^2\right)\phi_k\bigr|^2\d x\mathrm{d}y\\
	    &\le \sup_{\substack{x\in I_{3,n}\\n\in\mathcal{J}}} \left(V(x)+\lambda^2\right)^2 N(\lambda)\sum_{n\in\mathcal{J}}\int_{-\frac{5}{2}}^{\frac{5}{2}}\int_{I_{3,n}}\sum_{\lambda_k\le \lambda}\bigl| b_k\cosh(\lambda_k y)\phi_k\bigr|^2\d x\mathrm{d}y
	    \end{aligned}
	\end{equation}
	where the second line has used the fact that for each $k\in\Z$, $\phi_k$ is the eigenfunction for eigenvalue $\lambda_k^2$. Since the supremum of $V(x)+\lambda^2$ on $\bigcup_{n\in\mathcal{J}}I_{3,n}$ in the last line is bounded above by a large enough polynomial of $\lambda$,  we obtain 
	\begin{equation}
	    \begin{aligned}
	    \sum_{n\in\mathcal{J}}\|\chi_nD^2_x\Phi\|^2_{L^2(\R^2)}&\le \kappa_6\lambda^{\kappa_6} e^{5\lambda}\|\phi\|^2_{L^2(\R)}
	    \end{aligned}
	\end{equation}
	for some positive constants $\kappa_6>0$. 
\medskip

\noindent{\bf Estimate of $\rm S_5$.} This term can be written as
\begin{equation}
\begin{aligned}
    \mathrm{S}_5&\le 
    3\sum_{n\in\mathcal{J}}\left(\|(D_y^2\chi_n)\Phi\|^2_{L^2(\R^2)}+4\|(D_y\chi_n)D_y\Phi\|^2_{L^2(\R^2)}+\|\chi_nD_y^2\Phi\|^2_{L^2(\R^2)}\right).
    \end{aligned}
\end{equation}
The method to estimate the first term is the same as the one in Step~1, and the method to estimate the second and third terms is the same as the one in Step~3. Then it is easy to check that it is still bounded by $\kappa_7 a_n^2e^{5\lambda}$ for some positive constant $\kappa_7>0$.

Now the proof is completed by summing all the estimates and absorbing $a_n$ with $e^{6\lambda}$.
\end{proof}

\subsection{Proof of Theorem~\ref{spectral-inequality-2}}

Now we turn to the proof of Theorem~\ref{spectral-inequality-2}. 
Like the one of Theorem~\ref{spectral-inequality}, we decompose the real line into intervals $I_n$'s. However, we define it by a different recurrence formula
\begin{equation}
    x_{n+1}=x_n+L\rho(x_n),
\end{equation}
where $L>0$ and $\rho$ is given by~\eqref{new-rho} which we recall here
\begin{equation}
    \rho(x)=\min\left\lbrace(R\log\log\langle x\rangle-R^{-1})^{\frac{1}{2}}\frac{1}{\langle x\rangle^{\frac{\beta_2}{2}}},1\right\rbrace.
\end{equation}
The value of $R$ will be chosen later.
The rest of the definition of intervals $I_n$'s is the same as in Section~\ref{sec-2}.

Notice that for any $\eta>0$, it is always possible to choose a positive constant $C:=C(\eta,R)>0$ which depends only on $\eta$ and $R$, such that for all $n\in\N$
\begin{equation}
   C^{-1} \frac{L}{x_n^{\beta_2/2}}\le x_{n+1}-x_{n}\le C\frac{L}{x_n^{\beta_2/2-\eta}}.
\end{equation}
Combining with this relation and using a similar proof as for Lemma~\ref{growing-estimate-lemma}, we obtain the following lemma easily:
\begin{lemma}\label{new-growing-estimate-lma}
Let $\left\{x_n\right\}_{n\in\N}$ be the sequence given by the above. Then for any $\eta>0$, there exists a positive constant $C:=C(\eta,R)$ which depend only on $\eta$ and $R$, such that for all $n\in\N_{+}$, we have
\begin{equation}
    C^{-1} (Ln)^{\frac{1}{\beta_2/2+1}} \le x_n \le C (Ln)^{\frac{1}{\beta_2/2-\eta+1}}.
\end{equation}
\end{lemma}
We also have the following lemma similar to Corollary~\ref{propagation-crc}:
\begin{lemma}
Let $\gamma_n\in(0,\frac{1}{2})$,$n\in\Z$ and $V\in L^\infty_{\mathrm{loc}}(\R)$ be the potential as in Assumption~\ref{assump1}. Let $\Omega$ be a measurable subset of $\R$ such that $|\omega_n|=\gamma_n|I_n|$ for every $n\in\Z$. Then for any $\varepsilon>0$, there exists a positive number $R_0=R(\varepsilon,V)>0$ and a positive integer $n_0=n_0(R)>0$, such that for any $0<R<R_0$ and any real-valued $H^2_{\mathrm{loc}}$ solution $\Phi$ of~\eqref{2d-elliptic} in $\R^2$, we have
\begin{equation}\label{only-differences}
    \|\Phi\|_{L^2 (D_{1,n})}\le C \|\Phi\|^{\alpha_n}_{L^2(\omega_n)}\left(\sup_{D_{2,n}}|\Phi|^{1-\alpha}\right),
\end{equation}
where
\begin{equation}\label{new-relation}
    \alpha_n\ge  \frac{1}{\langle x_n\rangle^\varepsilon |\log\gamma_n|^2}\quad\text{and }\,\,\, C\le \langle x_n\rangle^\varepsilon
\end{equation}
for all $n$'s satisfying $|n|\ge n_0$.
\end{lemma}
\begin{proof}
The proof follows the same strategy as the one of  Corollary~\ref{propagation-crc}, we replace the scaling constant $a_n$ with
\begin{equation}
	a_n=\frac{|x_n|^{\frac{\beta_2}{2}}}{ R\log\log\langle x_n\rangle-R^{-1} }
\end{equation}
for $|n|$ sufficiently large so that $|x|>16$ in those intervals.
	By Assumption~\ref{assump1}, the new potential satisfies the condition
	\begin{equation}
		\widetilde{V}(x)\le c_2 \frac{1}{a_n^2}\left\langle \frac{x}{a_n} \right\rangle^{\beta_2}\le {C'}^2 \left(R\log\log\langle x_n\rangle -R^{-1}\right)^2
	\end{equation}
	for all $x\in a_nI_{3,n}$ and for all $|n|$ large enough, where $C'=C'(c_2,\beta_2)$ depends only on $c_2$ and $\beta_2$. Here we have also used the relation in Lemma~\ref{new-growing-estimate-lma} to approximately identify ${x}/{a_n}\in I_{3,n}$ with $x_n$. 
	
Using Proposition~\ref{propagation-prp} to~\eqref{reduced-2d-elliptic}, we obtain
\begin{equation}\label{h-2}
    		\|\Phi\|_{L^2(D_{1,n})}\le C' a_n^{\frac{\alpha_n-2}{2}}\|\Phi\|^{\alpha_n'}_{L^2(\omega_n)}\left( \sup_{D_{2,n}}|\Phi|^{1-\alpha_n'} \right),
\end{equation}
 and the constants are given by
 \begin{equation}\label{h-4}
    C\le \exp\left(d\exp\left(dC'R\log\log\langle x_n \rangle-dC'R^{-1}\right)\right)\le \langle x_n\rangle^{\varepsilon}
\end{equation}
and
\begin{equation}\label{h-3}
    \alpha_n\ge \frac{\exp\bigl(-d\exp(dC'R\log\log\langle x_n\rangle-dC'R^{-1})\bigr)}{|\log\gamma_n|^2}\ge  \frac{1}{\langle x_n\rangle^\varepsilon|\log\gamma_n|^2}.
\end{equation}
where we have chosen $R=R(d,C')$ sufficiently small so that $dC'R\le 1$ and $d\exp(-d C'R^{-1})\le \varepsilon$. Throughout the proof, we can choose such a positive integer $n_0$ which is only determined by the choice of $R$, such that~\eqref{h-4} and~\eqref{h-3} are fulfilled for all $n$'s satisfying $|n|\ge n_0$.
\end{proof}

Now we finish the proof of Theorem~\ref{spectral-inequality-2}.
\begin{proof}[Proof of Theorem~\ref{spectral-inequality-2}]
Following the same step as the proof of Theorem~\ref{spectral-inequality}, the only differences are the estimates of $\alpha_n$ and $C$ in~\eqref{only-differences}. Precisely speaking, we obtain the inequality as~\eqref{a-5}
with $a_n\in\mathcal{J}$ replaced by
\begin{equation}
    a_n=\frac{|x_n|^{\frac{\beta_2}{2}}}{(R\log\log\langle x_n\rangle-R^{-1})^{\frac{1}{2}}}\le C_2'\langle \lambda \rangle^{\frac{\beta_2}{\beta_1}}
\end{equation}
for $|n|$ sufficiently large
and the inequality as~\eqref{aa-9}
with $\alpha_{n_0}$ replaced by
\begin{equation}
    \alpha_{n_0}\ge  \frac{1}{\langle x_{n_0}\rangle^\varepsilon |\log\gamma|^2|x_{n_0}|^{2\tau}}\ge C_3'\frac{1}{|\log\gamma|^2\langle x_{n_0}\rangle^{2\tau+\varepsilon}} \ge  C_4'\frac{\langle\lambda\rangle^{-\frac{4\tau +2\varepsilon}{\beta_1}}}{|\log\gamma|^2},
\end{equation}
where $C_2',C_3' $ and $C_4'$ depend on $V$ and $\gamma$, and we also used the fact $x_{n_0}\lesssim_{c_1,\beta_1} \lambda^{\frac{2}{\beta_1}}$.

The growth of the constant $C$ in~\eqref{new-relation} can be neglected compared to the fast growth of $\alpha_{n_0}^{-1}$. Hence the only essential change is the lower bound of $\alpha_{n_0}$, \textit{i.e.},
\begin{equation}\label{replace-of-proof}
    \lambda^{\frac{4\tau }{\beta_1}} \Longrightarrow \lambda^{\frac{4\tau +2\varepsilon}{\beta_1}}.
\end{equation}
 Therefore, we finish the proof by doing the replacement~\eqref{replace-of-proof} and rewriting $2\varepsilon$ as $\varepsilon$. 
 \end{proof}

  \subsection{Proof of Theorem~\ref{spectral-inequality-3}}

 In this section, we show how to combine the method used to prove Theorem~\ref{spectral-inequality} with an auxiliary spectral inequality from more regular sets, to obtain the spectral inequality from thick sets with decaying density.
 
 The auxiliary spectral inequality that we mentioned is the following:
 \begin{proposition}\label{aux-spectral-inequality}
 Let $H=-\partial_x^2+V(x)$ and $\sigma\ge 0$. Assume that $V$ satisfies Assumption~\ref{assump2} and there exists a sequence $\lbrace x_k\rbrace_{k\in\Z}\subset \R$ such that $\Omega$ satisfies
 \begin{equation}\label{w-2}
     \Omega\cap \left(k+\left[0,L\right]\right) \supset I_{\langle k\rangle^{-\sigma}L}(x_i)
 \end{equation}
 for each $k\in\Z$ with some $L>0$. Then there exists a constant $C:=C(V,\Omega,\sigma)$ depending only on the parameters of $V$, $\Omega$ and $\sigma$ such that for any $\lambda>0$ and any $\phi\in\mathcal{E}_\lambda(\sqrt{H})$, we have
\begin{equation}\label{eqn-aux-spectral-inequality}
    \|\phi\|^2_{L^2(\R)}\le Ce^{C\bigl(\lambda^{\frac{\beta_2}{\beta_1}}+\lambda^{\frac{\beta_2}{\beta_1}}\log(\lambda+1)\bigr)}\|\phi\|^2_{L^2(\Omega\cap I_\lambda)}
\end{equation}
where $I_\lambda$ is the interval defined in~\eqref{bounded-interval}.
 \end{proposition}

 The proof of the above spectral inequality can be obtained from the proof of~\cite[Theorem~1]{zhu2023spectral}. In \cite{zhu2023spectral}, the authors obtained the spectral inequality for $H=-\partial_x^2+V(x)$ under Assumption~\ref{assump2} in any dimension $d$, from the set $\Omega$ satisfying the following property:
 for some $\gamma\in(0,1)$, there exists an equidistributed sequence $\lbrace z_k:k\in\Z^d\rbrace \subset \R^d$ such that
 \begin{equation}
     \Omega\cap (k+[0,L]^d)\supset B_{\gamma^{1+|k|^\sigma}}(z_k)
 \end{equation}
 for all $k\in\Z^d$. Indeed, the only essential change we need to do in the proof is to change the condition posed on the set $\Omega$ to the condition~\eqref{w-2}. For completeness, we will present the proof of Proposition~\ref{aux-spectral-inequality} in Appendix~\ref{appendix-aux}

 Let $\Omega$ be a measurable set satisfying~\eqref{eqn-by-Ming} with some $L>0$ and $\gamma\in[0,1)$. For simplicity, we assume $L=1$ in Theorem~\ref{spectral-inequality-3}. For other values of $L$, we can obtain it by scaling the equation.  Without loss of generality, we assume that $\lambda$ is large so that $|I_\lambda|$ is much larger than $1$. 
 
 For each $n\in\Z$, split $[n,n+1]$ into small intervals with the same length $\asymp \langle n\rangle ^{-\delta}$, where $\delta>0$ will be determined to be ${\beta_2}/2$ later.  Then by pigeonhole's principle, there exists at least one interval, denoted by $I'_n$, such that $|I'_n|\asymp \langle n\rangle ^{-\delta}$ and 
 \begin{equation}\label{w-1}
     |\Omega\cap I_n'|\gtrsim \gamma^{\langle n\rangle^\tau}|I_n'|.
 \end{equation}
Let $I_n'$ be chosen as above and set 
\begin{equation}
    \Omega_0=\bigcup_{n\in\Z} I_n'.
\end{equation}
Since  for all $n\in\Z$ we have $|I_n'|\asymp \langle n\rangle^{-\delta} $, it is easy to check that there exists a positive constant $\delta'>\delta$ such that
\begin{equation}
    \Omega_0\cap[n,n+1] \supset I_{\langle n\rangle^{-\delta'}}(x_n)
\end{equation}
where $x_n$ denotes the center of interval $I_n'$. 
Thus, by Proposition~\ref{aux-spectral-inequality}, we obtain
\begin{equation}\label{w-4}
    \|\phi\|^2_{L^2(\R)}\le Ce^{C\lambda^{\frac{\beta_2}{\beta_1}}\log (\lambda+1)}\|\phi\|^2_{L^2(\Omega_0\cap I_\lambda)}
\end{equation}
for all $\phi\in\mathcal{E}_\lambda(\sqrt{H})$.

For any $\phi\in\mathcal{E}_\lambda(\sqrt{H})$, we  lift it to a 2D function $\Phi$ by~\eqref{a-2}.  To proceed, we introduce some notations. We define for all $n\in\Z$
\begin{equation}
    a_n:=|I_n'|^{-1}\asymp \langle n\rangle^{\delta},
\end{equation}
\begin{equation}
    I_{j,n}:=(2j-1)I_n',\quad j=1,2,3,
\end{equation}
\begin{equation}
    D_{j,n}:= I'_{j,n}\times \left[-\frac{(2j-1)|I_n'|}{2},\frac{(2j-1)|I_n'|}{2}\right], \quad j=1,2,3.
\end{equation}
These are different from those defined in Subsection~\ref{subsec-4-2}. These will also be the case for several other notations below, allowing us to highlight the similarities with previous proofs.

\begin{proof}[Proof of Theorem~\ref{spectral-inequality-3}]
We proceed as in Subsection~\ref{subsec-3-4}. We fix $\displaystyle \delta=\frac{\beta_2}{2}$ and do the scaling transform as in~\eqref{transformation} to obtain~\eqref{reduced-2d-elliptic}. Then we have
\begin{equation}
    \widetilde{V}(x):= \frac{1}{a_n^2}V(\frac{x}{a_n})\lesssim a_n^{-2}n^{\beta_2}\lesssim 1
\end{equation}
for all $x/a_n\in I_n'\subset [n,n+1]$.
Applying Corollary~\ref{propagation-crc} to $\Phi$ with 
$\omega_n:=\Omega\cap I_n'$ and then by Young's inequality,
 we have for all $n\in\Z$ and all $\varepsilon>0$
	\begin{equation}\label{w-3}
		\|\Phi\|^2_{L^2(D_{1,n})}\le \frac{C_1 {a_n}^{\frac{\alpha_n-2}{2}} \alpha_n}{\varepsilon }\|\phi\|^2_{L^2(\omega_n)}+C_1{a_n}^{\frac{\alpha_n-2}{2}}(1-\alpha_n)\varepsilon ^{\frac{\alpha_n}{1-\alpha_n}}\|\Phi\|^2_{L^{\infty}(D_{2,n})},
	\end{equation}
	where the positive constant $C_1:=C_1(c_2,\beta_2)$ depends only on $c_2$ and $\beta_2$. Besides, $\alpha_n\in (0,1)$ satisfies
 \begin{equation}\label{asymp-w-1}
         \alpha\asymp \frac{1}{\big|\log\frac{|\omega_n|}{|I_n'|}\bigr|^2}\asymp \frac{1}{\langle n\rangle ^{2\tau}+\delta^2\log^2\langle n\rangle}\asymp \langle n\rangle ^{-2\tau}
 \end{equation}
which follows from~\eqref{w-1}. Define 
\begin{equation}
    \mathcal{J}:=\lbrace n \in\Z: [n,n+1]\cap I_\lambda\neq \varnothing\rbrace  \text{ and } n_0:=\max_{n\in\mathcal{J}}n.
\end{equation}
This is also different from the one we defined before. Notice that $I_n'\subset [n,n+1]$, clearly, we have
\begin{equation}
    n_0\le C\lambda^{2 /\beta_1}\label{w-6}
\end{equation}
from the fact $n\in I_\lambda$. Noticing that $a_n^\frac{\alpha_n-2}{2}\le 1$, summing~\eqref{w-3} over $\mathcal{J}$ gives 
\begin{equation}\label{w-5}
    	\sum_{n\in \mathcal{J}}\|\Phi\|^2_{L^{2}(D_{1,n})}\le \frac{C_1}{\varepsilon }\sum_{n\in \mathcal{J}}\|\phi\|^2_{L^2(\omega_n)}+C_1\varepsilon ^{\frac{\alpha_{n_0}}{1-\alpha_{n_0}}}\sum_{n\in \mathcal{J}} \|\Phi\|^2_{L^\infty(D_{2,n})}.
\end{equation}
The second term on the right-hand side is estimated the same way as~\eqref{a-7}. We only need to bound the term on the left-hand side from below. Indeed, we have
\begin{equation}
    \begin{aligned}    \sum_{n\in\mathcal{J}}\|\Phi\|^2_{L^2(D_{1,n})}&=\sum_{n\in\mathcal{J}}\int_{-\frac{1}{2}|I_n'|}^{\frac{1}{2}|I_n'|}\int_{I_n'}|\Phi|^2\d x\mathrm{d}y\\
    &\ge \int_{-\frac{1}{2}|I_{n_0}'|}^{\frac{1}{2}|I_{n_0}'|}\|\Phi(\cdot,y)\|^2_{L^2_x\bigl(\bigcup_{n\in\mathcal{J}} I_{n}'\bigr)}\d y\\
    &\ge \int_{-\frac{1}{2}|I_{n_0}'|}^{\frac{1}{2}|I_{n_0}'|} \|\Phi(\cdot,y)\|^2_{L^2_x(\Omega_0\cap I_\lambda)}\d y\\
    &\ge C_2\int_{-\frac{1}{2}|I_{n_0}'|}^{\frac{1}{2}|I_{n_0}'|} e^{-C_2\lambda^{\frac{\beta_2}{\beta_1}}\log \lambda}\|\Phi(\cdot,y)\|^2_{L^2_x(\R)}\d y,
    \end{aligned}
\end{equation}
where in the last step we used the spectral inequality in Proposition~\ref{aux-spectral-inequality}, since $\Phi(\cdot,y)\in \mathcal{E}_\lambda (\sqrt{H})$ for any fixed $y\in\R$. Then by the orthogonality of eigenfunctions $\phi_k$'s, we get
\begin{equation}
\begin{aligned}
    \int_{-\frac{1}{2}|I_{n_0}'|}^{\frac{1}{2}|I_{n_0}'|} \|\Phi(\cdot,y)\|^2_{L^2_x(\R)}\d y=&  \int_{-\frac{1}{2}|I_{n_0}'|}^{\frac{1}{2}|I_{n_0}'|} \sum_{\lambda_k\le \lambda}|b_k|^2|\cosh(\lambda_k y)|^2\d y\\
    \ge |I_{n_0}|\|\phi\|^2_{L^2(\R)}.
    \end{aligned}
\end{equation}
Also notice that $I'_{n_0}\asymp n_0^{-\delta}$, by~\eqref{w-6}, we find $|I'_{n_0}|\gtrsim \lambda^{-\frac{2\delta}{\beta_1}}$. Hence
\begin{equation}
     \sum_{n\in\mathcal{J}}\|\Phi\|^2_{L^2(D_{1,n})}\ge C_3 \lambda^{-\frac{2\delta}{\beta_1}}e^{-C_3\lambda^{\frac{\beta_2}{\beta_1}}\log (\lambda+1)}\|\phi\|^2_{L^2(\R)}.\label{w-7}
\end{equation}
Inserting~\eqref{a-7} and~\eqref{w-7} into~\eqref{w-5}, we obtain
\begin{equation}\label{w-8}
    \|\phi\|^2_{L^2(\R)}\le \frac{C_4}{\varepsilon} e^{C_4\lambda^{\frac{\beta_2}{\beta_1}}\log (\lambda+1)} \|\phi\|^2_{L^2(\Omega)}+C_4\varepsilon^{\frac{\alpha_{n_0}}{1-\alpha_{n_0}}} e^{C_4\lambda^{\frac{\beta_2}{\beta_1}}\log (\lambda+1)}\|\phi\|^2_{L^2(\R)}
\end{equation}
where we absorbed the polynomials of $\lambda$ and used the fact that $\displaystyle \frac{\beta_2}{\beta_1}\ge 1$ to absorb the term $e^{C\lambda}$ in~\eqref{a-7}.

Now choose
\begin{equation}
    \varepsilon=\left(\frac{1}{2C_4 e^{C_4 \lambda^{\frac{\beta_2}{\beta_1}}\log (\lambda+1)}}\right)^{\frac{1-\alpha_{n_0}}{\alpha_{n_0}}} \Longleftrightarrow C_4\varepsilon^{\frac{\alpha_{n_0}}{1-\alpha_{n_0}}} e^{C_4\lambda^{\frac{\beta_2}{\beta_1}}\log (\lambda+1)}=\frac{1}{2}
\end{equation}
in~\eqref{w-8}, we get
\begin{equation}
    \frac{1}{\varepsilon}\le \left( 2C_4 e^{C_4\lambda^{\frac{\beta_2}{\beta_1}}\log (\lambda+1)} \right)^{\frac{1}{\alpha_{n_0}}}\le C_5e^{C_5\lambda^{\frac{4\tau }{\beta_1}+\frac{\beta_2}{\beta_1}}\log(\lambda+1)}
\end{equation}
where $C_5:=C_5(V)>0$ depends only on $V$ and we also used the fact $\alpha_{n_0}\asymp \langle n_0\rangle^{-2\tau}\gtrsim \lambda^{-\frac{4\tau}{\beta_1}}$ from~\eqref{asymp-w-1} and ~\eqref{w-6}. Substituting this $\varepsilon$ into~\eqref{w-8}, we obtain
\begin{equation}
    \|\phi\|^2_{L^2(\R)}\le \frac{2C_4}{\varepsilon} e^{C_4\lambda^{\frac{\beta_2}{\beta_1}}\log(\lambda+1)}\|\phi\|^2_{L^2(\Omega)}\le C'e^{C'\lambda^{\frac{4\tau }{\beta_1}+\frac{\beta_2}{\beta_1}}\log(\lambda+1)}\|\phi\|^2_{L^2(\Omega)}
\end{equation}
where $C'=C'(V)$ depends only on $V$. This finishes the proof of Theorem~\ref{spectral-inequality-3}.
\end{proof}

\appendix 

\section{Proof of Proposition~\ref{aux-spectral-inequality}}\label{appendix-aux}

We follow the strategy in \cite{zhu2023spectral} to prove Proposition~\ref{aux-spectral-inequality} in any $d$-dimension, hence we always assume that the potential $V=V_1+V_2$ satisfies Assumption~\ref{assump2}. Without loss of generality, we can set $L=1$ in Proposition~\ref{aux-spectral-inequality}, and change the density condition~\eqref{w-2} to the following more symmetric one: there exists a sequence $\{x_k\}_{k\in\Z^d}\subset \R^d$ such that $\Omega$ satisfies
\begin{equation}\label{w-2-new}
    \Omega\cap \left( k+\left[-\frac{1}{2},\frac{1}{2}\right]^d \right)\supset I_{\langle k\rangle ^{-\sigma}}(x_k)
\end{equation}
for each $k\in\Z^d$ with some $\sigma\ge 0$.

We first introduce some notations. Given $N>0$, we define
\begin{equation}
    \mathcal{Q}_N:=\left[ -\frac{N}{2},\frac{N}{2} \right]^d.
\end{equation}
We denote by $\mathcal{B}_r(x)\subset \R^d$ the $d$-dimensional ball with radius $r$ and center $x\in\R^d$, by $\mathbb{B}_r(x)\subset\R^{d+1}$ the $(d+1)$-dimensional ball with radius $r$ and center $x$.

Let $\delta\in (0,\frac{1}{2})$, $b=(0,\cdots,0,-b_{d+1})\in \R^{d+1}$ and $b_{n+1}=\frac{\delta}{100}$. Define
\begin{equation}
    W_1=\left\lbrace y\in \R_{+}^{d+1}: |y-b|\le \frac{1}{4}\delta \right\rbrace
\end{equation}
and
\begin{equation}
    W_2=\left\lbrace y\in \R_{+}^{d+1}: |y-b|\le \frac{2}{3}\delta  \right\rbrace
\end{equation}
so that $W_1\subset W_2\subset \mathbb{B}_\delta (b)$. Write $\Lambda_N:=\mathcal{Q}_N\cap \Z^{d}$, and define
\begin{equation}
    W_j(z_i):=(z_i,0)+W_j,\quad j=1,2,
\end{equation}
as well as
\begin{equation}
    P_j(N)=\bigcup_{i\in \Lambda_N} W_j(z_i),\quad j=1,2, \text{ and } D_\delta(N)=\bigcup_{i\in\Lambda_N}\mathcal{B}_\delta(z_i).
\end{equation}
Define $R=9\sqrt{d}$ and 
\begin{equation}
    X_1=\mathcal{Q}_N\times [-1,1] \text{ and } \widetilde{X}_R=\mathcal{Q}_{N+R}\times[-R,R].
\end{equation}

We give two three-ball inequalities that are needed in the proof.
The first three-ball inequality in $\R^{d+1}$ we need is the following:
\begin{lemma}[{\cite[Lemma~1]{zhu2023spectral}}]\label{lma-3-ball-1}
There exist $0<\alpha<1$ and $C>0$ depending only on $d$ such that, if $v$ is the solution of 
\begin{equation}
\begin{cases}
-(\Delta_x+\partial_{x_{d+1}}^2)v(x,x_{d+1})+V(x)v(x,x_{d+1})=0, & x\in\R^d,x_{d+1}\in\R,\\
v(x,0)=0, &x\in\R^d,
\end{cases}
\end{equation}
then 
\begin{equation}\label{3-ball-1}
    \|v\|_{H^1(P_1(N))} \le \delta^{-\alpha}\!\exp\left( C\left(1+\mathcal{G}(V_1,V_2,9\sqrt{d}N)\right) \right) \|v\|^\alpha_{ H^1(P_2(N))}\left\|\frac{\partial v}{\partial y_{d+1}}\right\|^{1-\alpha}_{ L^2(D_\delta(N))}\!\!,
\end{equation}
where
\begin{equation}\label{def-mca-g}
    \mathcal{G}(V_1,V_2,N)=\|V_1\|^{\frac{1}{2}}_{W^{1,\infty}(\mathcal{Q}_N)}+\|V_2\|^{\frac{2}{3}}_{L^\infty (\mathcal{Q}_N)}.
\end{equation}
\end{lemma}

The second three-ball inequality in $\R^{d+1}$ is:
\begin{lemma}[{\cite[Lemma~2]{zhu2023spectral}}]\label{lma-3-ball-2}
    Let $\delta\in (0,\frac{1}{2})$. There exist $C>0$ depending only on $d$, $0<\alpha<1$ depending only on $\delta$ and $d$ such that, if $v$ is the solution of 
    \begin{equation}
        \begin{cases}
            -(\Delta_x+\partial^2_{x_{d+1}})v(x,x_{d+1})+V(x)v(x,x_{d+1})=0,& x\in\R^d, x_{d+1}\in \R,\\
            v(x,-y)=-v(x,y) & x\in \R^d, y \in\R,
        \end{cases}
    \end{equation}
    then 
    \begin{equation}\label{3-ball-2}
        \|v\|_{H^1(X_1)}\le \delta^{-\alpha_1}\exp\left( C \left( \mathcal{G}( V_1,V_2,9\sqrt{d}N) \right) \right)\|v\|^{1-\alpha_1}_{H^1(\widetilde{X}_R)}\|v\|^{\alpha_1}_{H^1(P_1(N))},
    \end{equation}
    where $\mathcal{G}(V_1,V_2,N)$ is given by~\eqref{def-mca-g}. Furthermore, $\alpha_1$ can be given in the form
    \begin{equation}\label{3.7}
        0<\alpha_1=\frac{\epsilon_1}{|\log\delta|+\epsilon_2}<1
    \end{equation}
with positive constants $\epsilon_1$ and $\epsilon_2$ depending only on $d$.
\end{lemma}

For any $\phi\in\mathcal{E}_\lambda(\sqrt{H})$, it can be written as $\displaystyle\phi_k(x)=\sum_{\lambda_k\le\lambda}b_k\phi_k(x)$. We use a different ghost dimension construction: we set
\begin{equation}\label{different-lift}
    \Phi(x,y):=\sum_{\lambda_k\le \lambda} b_k \frac{\sinh(\lambda_k y)}{\lambda_k}\phi_k(x).
\end{equation}
Then $\Phi(x,y)$ satisfies the equation
\begin{equation}
    -\Delta\Phi +V\Phi=0.
\end{equation}
where $\Delta=\Delta_x+\partial_y^2$. It is easy to check that
\begin{equation}
    \partial_y\Phi(x,0)=\phi(x) \text{ and } \Phi(x,0)=0.
\end{equation}
The following estimate for $\Phi$ is standard and can be found in \cite[Lemma~3]{zhu2023spectral}:
\begin{lemma}\label{lma3.3}
Let $\phi\in\mathcal{E}_\lambda(\sqrt{H})$ and $\Phi$ be given in \eqref{different-lift}. Then for any $\rho>0$, we have
\begin{equation}
    2\rho \|\phi\|^2_{L^2(\R^d)} \le \|\Phi\|^2_{H^1(\R^d\times (-\rho,\rho))}\le 2\rho \left(1+\frac{\rho^2}{3}(1+\lambda^2)e^{2\rho\lambda}\right)\|\phi\|^2_{L^2(\R^d)}.
\end{equation}
\end{lemma}
Based on Lemma~\ref{localization}, we obtain the following corollary:
\begin{corollary}
    Given the same condition as in Lemma~\ref{localization}, we have
    \begin{equation}\label{3.11}
        \|\Phi\|^2_{H^1(\R^d\times(-1,1))}\le 2\|\phi\|^2_{H^1( \mathcal{B}_r(0) )}.
    \end{equation}
\end{corollary}

\begin{proof}
	Since $\Phi(\cdot ,x_{d+1})\in \mathcal{E}_{\lambda}(\sqrt{H})$, by  Lemma~\ref{localization} we obtain
	\begin{equation}\label{3.12}
\|\Phi(\cdot ,x_{d+1})\|^2_{H^{1}(\R^{d})}
\le 2\|\Phi(\cdot ,x_{d+1})\|^2_{L^2(\mathcal{B}_r(0) ) }
\le 2 \|\Phi(\cdot ,x_{d+1})\|^2_{H^1( \mathcal{B}_r(0) ) }.
	\end{equation}
	Since $\partial_{d+1}\Phi(\cdot ,x_{d+1})\in \mathcal{E}_{\lambda}(H)$, we obtain
	\begin{equation}\label{3.13}
		\|\partial_{d+1}\Phi(\cdot ,x_{d+1})\|^2_{L^2(\R^{d})}
		\le \|\partial_{d+1}\Phi(\cdot ,x_{d+1})\|^2_{H^{1}(\R^{d})}
		\le 2 \|\partial_{d+1}\Phi\|^2_{L^2( \mathcal{B}_r(0) ) }.
	\end{equation}
Then we have
\begin{equation*}
	\begin{aligned}
\big\|\Phi\big\|^2_{H^{1}(\R^{d}\times (-1,1))}
=& \int_{-1}^{1}\big\|\Phi(\cdot ,x_{d+1})\big\|^2_{L^2(\R^{d})}\d x_{d+1}
		+ \int_{-1}^{1}\sum_{j=1}^{n} \big\|\partial_j \Phi(\cdot ,x_{d+1})\big\|^2_{L^2(\R^{d})}\d x_{d+1}\\
		 &+ \int_{-1}^{1}\big\|\partial_{d+1}\Phi(\cdot ,x_{d+1})\big\|^2_{L^2(\R^{d})}\d x_{d+1}\\
\le& 
\int_{-1}^{1}\big\|\Phi(\cdot ,x_{d+1})\big\|^2_{H^{1}(\R^{d})}\d x_{d+1} 
+ \int_{-1}^{1}2\big\|\partial_{d+1}\Phi(\cdot ,x_{d+1})\big\|^2_{L^2(\mathcal{B}_r(0))}\d x_{d+1}\\
	\end{aligned} 
\end{equation*}
with \eqref{3.13}. Using \eqref{3.12} we then obtain
\begin{equation*}
	\begin{aligned}
\big\|\Phi\big\|^2_{H^{1}(\R^{d}\times (-1,1))}
\le 
& \int_{-1}^{1}2 \big\|\Phi(\cdot ,x_{d+1})\big\|^2_{H^{1}(\mathcal{B}_r(0))}\d x_{d+1}
+\int_{-1}^{1}2\big\|\partial_{d+1}\Phi(\cdot ,x_{d+1})\big\|^2_{L^2(\mathcal{B}_r(0))}\d x_{d+1}\\
		=& 2\big\|\Phi\big\|^2_{H^{1}( \mathcal{B}_r(0)) }
	\end{aligned} 
\end{equation*}
as claimed.
\end{proof}

We can now prove Proposition~\ref{aux-spectral-inequality}:
\begin{proof}[Proof of Proposition~\ref{aux-spectral-inequality}] 
Let $N=2\lceil r\rceil+1$, where $r=C\lambda^{\frac{2}{\beta_1}}$ is given in Lemma~\ref{localization} and $\lceil a \rceil$ means the largest integer smaller than $a+1$. Then we have $\mathcal{B}_r(0)\subset \mathcal{Q}_N$. Moreover, we can decompose $\mathcal{Q}_N$ as
\begin{equation}
    \mathcal{Q}_N=\bigcup_{k\in\Lambda_N}\left(k+\left[-\frac{1}{2},\frac{1}{2}\right]^d\right).
\end{equation}
For each $k\in\Lambda_N$, we have $|k|\le \sqrt{d}\lceil r\rceil$. As $\gamma\in(0,\frac{1}{2})$, we get
\begin{equation}
    \delta:=\langle r\rangle ^{-\sigma}\lesssim \langle k\rangle^{-\sigma}
\end{equation}
for all $k\in\Lambda_N$.

Now we show an interpolation inequality. Notice that $\Phi$ is odd in $x_{d+1}$, so taking $v=\Phi$, we combine~\eqref{3-ball-1} in Lemma~\ref{lma-3-ball-1} and~\eqref{3-ball-2} in Lemma~\ref{lma-3-ball-2} with $\delta$ and $N$ defined above to get
\begin{equation}
\begin{aligned}
\!\|\Phi\|_{H^{1}(X_1)}
&
\le \delta^{-2 \alpha_1}\exp\Big( C\big( 1+\mathcal{G}( V_1,V_2,9 \sqrt{d} N)\big)\Big) 
\big\|\Phi\big\|^{1-\alpha_1}_{H^{1}(\widetilde{X}_{R})}
\big\|\Phi\big\|^{\alpha_1}_{H^{1}( P_1(N)) }\\
&\le \delta ^{-2 \alpha_1-\alpha \alpha_1}\exp\Big( C\big( 1+\mathcal{G}( V_1,V_2,9 \sqrt{d} N)\big)\Big) 
\big\|\Phi\big\|^{\alpha \alpha_1}_{H^{1}(P_2(N))}
\left\|\frac{\partial \Phi}{\partial y_{n+1}}\right\|^{\alpha_1  (1-\alpha )}_{L^2(D_\delta(N))}\!
\big\|\Phi\big\|^{1-\alpha_1}_{H^{1}(\widetilde{X}_{R})}\\
& \le \delta^{-3 \alpha_1}\exp\Big( C\big( 1+\mathcal{G}( V_1,V_2,9 \sqrt{d} N)\big)\Big) 
\big\|\phi\big\|^{\hat{\alpha}}_{L^2\big(D_\delta(N)\big)}
\big\|\Phi\big\|^{1-\hat{\alpha}}_{H^{1}\left( \widetilde{X}_{R} \right) },
		\end{aligned} 
	\end{equation}
where $\hat{\alpha}=\alpha_1 (1-\alpha )$ and we have used the facts $\displaystyle P_2(N) \subset \widetilde{X}_{R}$ and $\dfrac{\partial \Phi}{\partial y_{n+1}}(\cdot ,0)=\phi$. Here and below, the symbol $C:=C(d,V)$ may represent different positive constants depending on $d$ and $V$.

Recall $\alpha_1$ in \eqref{3.7}, we have $\alpha_1 \asymp \hat{\alpha}\asymp \frac{1}{|\log \delta|}$ for any $\delta \in (0, \frac{1}{2})$. Hence $\delta ^{-3 \alpha_1}\le C$ and then
	\begin{equation}\label{3.18}
		\big\|\Phi\big\|_{H^{1}\left( X_1 \right) }
		\le \exp \Big( C\big( 1+\mathcal{G}( V_1,V_2,9 \sqrt{d} N) \big) \Big) 
		\big\|\phi\big\|^{\hat{\alpha}}_{L^2\left( \Omega \cap \mathcal{Q}_N \right) }\big\|\Phi\big\|^{1-\hat{\alpha}}_{H^{1}(\widetilde{X}_{R})},
	\end{equation}
where we have also used the fact $D_\delta (N)\subset \Omega \cap \mathcal{Q}_N$.

Substituting $N=2\lceil r\rceil+1$ into~\eqref{def-mca-g} and by Assumption~\ref{assump2}, we have
\begin{equation}
    \mathcal{G}(V_1,V_2,9\sqrt{d}N)\le C(1+\frac{9}{2}dN)^{\frac{\beta_2}{2}} \lesssim_{d,V} \langle \lambda\rangle ^{\frac{\beta_2}{\beta_1}}
\end{equation}
for $\lambda>1$. We can then write~\eqref{3.18} as
\begin{equation}\label{3.20}
    \|\Phi\|_{H^1(X_1)}\le \exp\bigl(C_*\lambda^{\frac{\beta_2}{\beta_1}}\bigr)\|\phi\|^{\hat{\alpha}}_{L^2(\Omega\cap \Lambda_N)}\|\Phi\|^{1-\hat{\alpha}}_{H^1(\widetilde{X}_{R})}.
\end{equation} 
We now bound $\|\Phi\|^2_{H^1(\R^d\times (-\rho,\rho))}$ from above and below by respectively taking $\rho=R$ and $\rho=1$ in Lemma~\ref{lma3.3}. This gives
\begin{equation}
    \frac{\|\Phi\|^2_{H^1(\R^d\times (-R,R))}}{\|\Phi\|^2_{H^1(\R^d\times (-1,1))}}\le R\left( 1+\frac{R^2}{3}(1+\lambda^2) \right)\exp(2R\lambda)\le \exp\left(C_2\lambda\right).
\end{equation}
With the aid of~\eqref{3.11} and $\mathcal{B}_r(0)\subset \mathcal{Q}_N$, we get
\begin{equation}\label{3.22}
	\begin{aligned}
\big\|\Phi\big\|_{H^{1}( \R^{d}\times (-R,R) ) } 
&\le \exp\left( \frac{1}{2}C_2 \lambda  \right) \big\|\Phi\big\|_{H^{1}( \R^{d}\times (-1,1) ) }\\
&\le \sqrt{2} \exp \left( \frac{1}{2}C_2 \lambda  \right)
\big\|\Phi\big\|_{H^{1}(\mathcal{Q}_N\times (-1,1) ) }.
	\end{aligned} 
\end{equation}
Recall that $X_1=\mathcal{Q}_N\times (-1,1)$, substituting~\eqref{3.22} into  into~\eqref{3.20} we obtain
\begin{equation}
    \|\Phi\|_{H^1(\R^d\times(-R,R))}\le \exp\left(C_3\lambda^{\frac{\beta_2}{\beta_1}}\right)\|\phi\|^{\hat{\alpha}}_{L^2(\Omega\cap\mathcal{Q}_N)}\|\Phi\|^{1-\hat{\alpha}}_{H^1(\widetilde{X}_R)},
\end{equation}
where we also used the fact $\displaystyle\frac{\beta_2}{\beta_1}\ge 1$. Since $\widetilde{X}_R\subset \R^d\times (-R,R)$, it follows that
\begin{equation}
    \|\Phi\|_{H^1(\R^d\times(-R,R))}\le \exp \left(\hat{\alpha}^{-1}C_3\lambda^{\frac{\beta_2}{\beta_1} }\right)\|\phi\|_{L^2(\Omega\cap\mathcal{Q}_N)}.
\end{equation}
Recall that 
\begin{equation}
    \hat{\alpha}^{-1}\asymp \alpha_1^{-1}\asymp |\log\delta|\lesssim \biggl(1+ \frac{2\sigma}{\beta_1}\log (\lambda+1)\biggr),
\end{equation}
we obtain
\begin{equation}
    \|\Phi\|_{H^1(\R^d\times(-R,R))}\le e^{C\bigl(\lambda^{\frac{\beta_2}{\beta_1}}+\frac{2\sigma}{\beta_1}\lambda^{\frac{\beta_2}{\beta_1}}\log(\lambda+1)\bigr)}\|\phi\|_{L^2(\Omega\cap\mathcal{Q}_{N})}.
\end{equation}
Finally, using the lower bound in Lemma~\ref{lma3.3} with $\rho=R$, we obtain
\begin{equation}
    \|\phi\|_{L^2(\R^d)}\le \left(\frac{1}{2R}\right)^{\frac{1}{2}}\|\Phi\|_{H^1(\R^d\times(-R,R))}\le Ce^{C\bigl(\lambda^{\frac{\beta_2}{\beta_1}}+\frac{2\sigma}{\beta_1}\lambda^{\frac{\beta_2}{\beta_1}}\log(\lambda+1)\bigr)}\|\phi\|_{L^2(\Omega\cap \mathcal{Q}_N)}
\end{equation}
where $C:=C(d,V)$ is a positive constant depending only on $d$ and $V$.
\end{proof}

\section*{Acknowledgements}
The author warmly thanks Philippe Jaming for his constant guidance. He is also particularly grateful to Ming Wang for his suggestive advice and idea about using Zhu-Zhuge's spectral inequality. The author also thanks Pei Su, Chenmin Sun and Yuzhe Zhu for helping me understand deeply the techniques used in the article. It was supported by China Scholarship Council grant CSC202206410007.

\bibliographystyle{alpha}
\bibliography{mybib}

\end{document}